\documentclass[1p,12pt]{elsarticle}
\usepackage{graphicx}
\usepackage{amsmath,amsxtra,amssymb,latexsym, amscd,amsthm}
\usepackage[colorlinks=true,citecolor=black,linkcolor=black,urlcolor=blue]{hyperref}

\theoremstyle{plain}
\newtheorem{thm}{Theorem}[section]
\newtheorem{cor}[thm]{Corollary}
\newtheorem{lem}[thm]{Lemma}
\newtheorem{prop}[thm]{Proposition}
\theoremstyle{definition}

\theoremstyle{remark}
\newtheorem{rmk}{Remark}

\numberwithin{equation}{section}
\numberwithin{figure}{section}
\numberwithin{table}{section}

\newcommand{\wt}{\operatorname{wt}}

\newcommand{\M}{\operatorname{M}}

\DeclareMathOperator{\AR}{AR}
\begin{document}
\setlength{\baselineskip}{13truept}
\title{Enumeration of Hybrid Domino-Lozenge Tilings}
\author{TRI LAI}
\address{Indiana University\\
Department of Mathematics\\
 Bloomington, IN 47405, USA}
\date{\today}

\begin{abstract}
We solve and generalize an open problem posted by James Propp (Problem 16 in \emph{New Perspectives in Geometric Combinatorics},  Cambridge University Press, 1999) on the number of tilings of  quasi-hexagonal regions on the square lattice with every third diagonal drawn in. We also obtain a generalization of Douglas' Theorem on the number of tilings of a family of regions of the square lattice with every second diagonal drawn in.
\end{abstract}

\maketitle

\section{Introduction}

The field of exact enumeration of tilings (equivalently, perfect matchings) dates back to the early 1900's when MacMahon proved his classical theorem on the number of plane partitions that fit in a given box (see \cite{McM}). This theorem is equivalent to the fact that the number of unit rhombus (or lozenge) tilings of a hexagon $H_{a,b,c}$ of side-lengths $a$, $b$, $c$, $a$, $b$, $c$ (in cyclic order) drawn on the triangular lattice is equal to
\begin{equation}\label{McMahon}
\M(H_{a,b,c})=\prod^{a}_{i=1}\prod^{b}_{j=1}\prod^{c}_{k=1}\frac{i+j+k-1}{i+j+k-2}
\end{equation}
(we use the operator $\M$ to denote both the number of tilings of a lattice region and the number of perfect matchings of a graph, as the two objects can be identified by a well known bijection).

Another classical result, from the early 1960's, is the enumeration of domino tilings of a rectangle on the square lattice, due independently to Kasteleyn \cite{Kas} and Temperley and Fisher \cite{Fish}. This states that the number of domino tilings of a $2m$ by $2n$ rectangle on the square lattice equals
\begin{equation}\label{Kasteleyn}
\M(G_{2m,2n})=2^{2mn}\prod_{j=1}^m\prod_{k=1}^n\bigg(\cos^2\Big(\frac{j\pi}{2m+1}\Big)+\cos^2\Big(\frac{k\pi}{2n+1}\Big)\bigg).
\end{equation}

\begin{figure}\centering
\includegraphics[width=7cm]{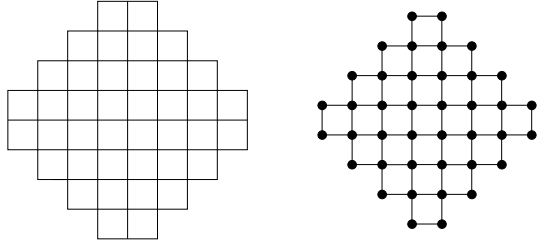}
\caption{The Aztec diamond region (left) and  the Aztec diamond graph (right) of order 4.}
\label{aztecdiamond}
\end{figure}

In the early 1990's, Elkies, Kuperberg, Larsen and Propp \cite{Elkies} considered another family of simple regions on the square lattice called Aztec diamonds (see Figure \ref{aztecdiamond} for an example), and proved that the number of domino tilings of the Aztec diamond of order $n$ is given by the simple formula
\begin{equation}\label{diamond}
\M(AD_n)=2^{n(n+1)/2}.
\end{equation}

A large body of related work followed (see e.g. \cite{Ciucu2},  \cite{Ciucu5}, \cite{Cohn}, \cite{Doug}, \cite{Gessel}, \cite{Yang}, and the references in \cite{Propp} for a more extensive list), centered on families of lattice regions whose tilings are enumerated by simple product formulas. The state of affairs at the end of that decade is captured by Propp's paper \cite{Propp} published in 1999, which presented a list of 32 open problems in the field of enumeration of perfect matchings.

Most of those 32 problems have been solved in the meanwhile, but some are still open. In this paper we solve and generalize one of these open problems (Problem 16 on Propp's list \cite{Propp}). Our methods also provide a new proof and a generalization for a related result of Douglas \cite{Doug}.

\section{Statement of main results}

Problem 16 on Propp's list \cite{Propp} concerns a family of quasi-hexagonal regions on the lattice obtained from the square lattice by drawing in every third southwest-to-northeast diagonal. The case when the side-lengths of the quasi-hexagon are $3$, $2$, $2$, $3$, $2$, $2$ (clockwise from top) is illustrated in Figure \ref{hexagonp}. Problem 16 of \cite{Propp} asks for a formula for the number of tilings of the quasi-hexagon of sides $a$, $b$, $c$, $a$, $b$, $c$ (the sides of length $a$ are the ones along \emph{diagonals}\footnote{From now on, ``diagonal(s)" refers to  ``southwest-to-northeast diagonal(s)".} of the square lattice), where the allowed tiles are unions of two fundamental regions of the resulting dissection of the square lattice sharing an edge. 

\begin{figure}\centering
\begin{picture}(0,0)%
\includegraphics{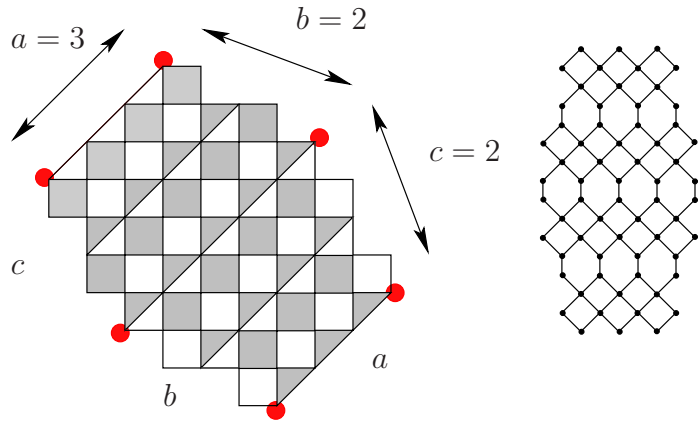}%
\end{picture}%
%
%
\setlength{\unitlength}{3947sp}%
\begingroup\makeatletter\ifx\SetFigFont\undefined%
\gdef\SetFigFont#1#2#3#4#5{%
  \reset@font\fontsize{#1}{#2pt}%
  \fontfamily{#3}\fontseries{#4}\fontshape{#5}%
  \selectfont}%
\fi\endgroup%
\begin{picture}(4289,2665)(222,-1852)
\put(237,485){\makebox(0,0)[lb]{\smash{{\SetFigFont{12}{14.4}{\familydefault}{\mddefault}{\updefault}{$a=3$}%
}}}}
\put(2009,603){\makebox(0,0)[lb]{\smash{{\SetFigFont{12}{14.4}{\familydefault}{\mddefault}{\updefault}{$b=2$}%
}}}}
\put(2836,-224){\makebox(0,0)[lb]{\smash{{\SetFigFont{12}{14.4}{\familydefault}{\mddefault}{\updefault}{$c=2$}%
}}}}
\put(2481,-1523){\makebox(0,0)[lb]{\smash{{\SetFigFont{12}{14.4}{\familydefault}{\mddefault}{\updefault}{$a$}%
}}}}
\put(1182,-1759){\makebox(0,0)[lb]{\smash{{\SetFigFont{12}{14.4}{\familydefault}{\mddefault}{\updefault}{$b$}%
}}}}
\put(237,-933){\makebox(0,0)[lb]{\smash{{\SetFigFont{12}{14.4}{\familydefault}{\mddefault}{\updefault}{$c$}%
}}}}
\end{picture}%
\caption{An example of the quasi-hexagonal regions considered in \cite{Propp}, Problem 16 (left), and its dual graph (right).}
\label{hexagonp}
\end{figure}

As mentioned in \cite{Propp}, the case $a=b=c$ (to which also the cases $a=b< c$ and, by symmetry,  $a=c<b$ turn out to reduce) has been solved by Ben Wieland (in unpublished work), who showed that the number of tilings is given in these situations by powers of 2.

In this paper we prove a simple product formula in the case $b=c$, which was previously open. The result turns out to fully justify Propp's initial motivation, stated in \cite{Propp}: ``One reason for my special interest in Problem 16 is that it seems to be a genuine hybrid of domino tilings of Aztec diamonds and lozenge tilings of hexagons.'' Indeed, we show that by a sequence of graph transformations --- which can also be used to prove the Aztec diamond theorem (\ref{diamond}) --- the dual graph can be transformed into a honeycomb graph whose number of perfect matchings is given by MacMahon's formula (\ref{McMahon}).

We in fact enumerate the tilings of more general regions, which we describe next. Rather than considering quasi-hexagonal regions on the square lattice on which diagonals are drawn in so that the distance\footnote{The unit here is the distance between two consecutive (southwest-to-northeast) diagonals of the square lattice, i.e. $\sqrt{2}/2$.} between any two successive ones is 3, we allow the general situation when the distances between successive diagonals are \emph{arbitrary}. Let $\ell$ be a fixed diagonal (this will contain the middle two vertices of the quasi-hexagon), and assume that $k$ diagonals have been drawn in above it, with the distances between successive ones, starting from top, being $d_1,\dotsc,d_k$ ($d_k$ is the distance between the bottommost of these $k$ diagonals and $\ell$). Assume also that $l$ diagonals have been drawn in below $\ell$,  with the distances between successive ones, starting from the bottom, being $d'_1,\dotsc,d'_l$ ($d'_l$ is the distance between the topmost of these $l$ diagonals and $\ell$).

Given a positive integer $a$, we define the quasi-hexagonal region \[H_a(d_1,\dotsc,d_k;d'_1,\dotsc,d'_l)\] as follows (see Figure \ref{hexagonregion} for an example). Its top and bottom boundaries are along the top and bottom diagonals that have been drawn in. Its southwestern and northeastern boundaries are defined in the next two paragraphs.

Color the resulting dissection of the square lattice black and white so that any two fundamental regions that share an edge have opposite color, and assume that the triangles just below the top diagonal are white. Let $A$ be a lattice point on the top diagonal. Start from $A$ and take unit steps south or east so that for each step the color of the fundamental region on the left is black. We arrive $\ell$  at a lattice point $B$. Continue downward from $B$ in similar fashion, with the one difference that now we require the fundamental region on the left to be white for each step. Let $C$ be the lattice point on the bottom diagonal that is reached by these steps. The described path from $A$ to $C$ is the southwestern boundary of our region.
\begin{figure}\centering
\begin{picture}(0,0)%
\includegraphics{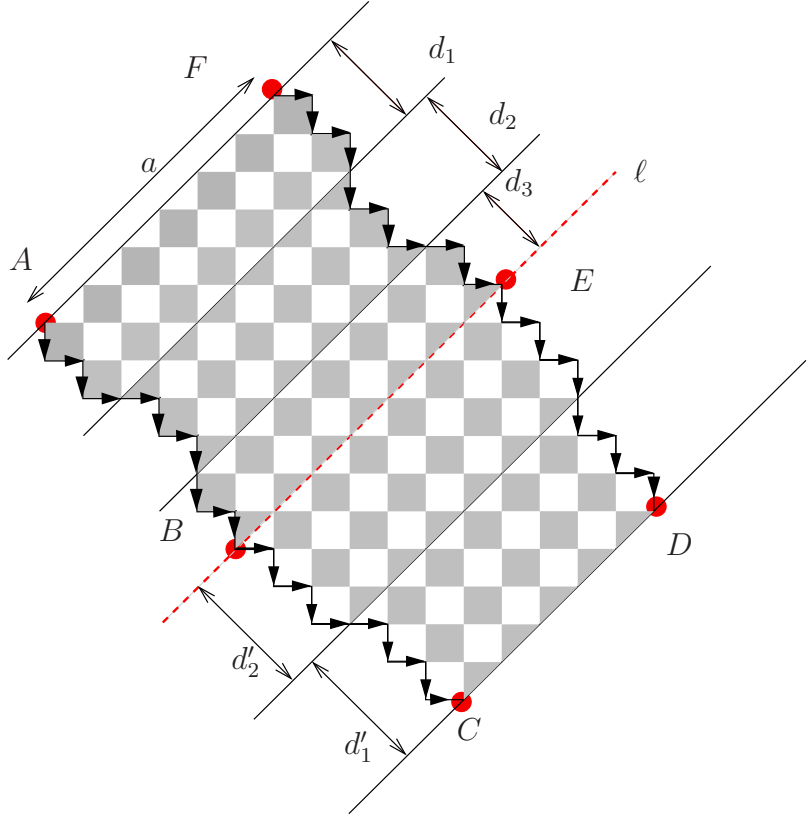}%
\end{picture}%
\setlength{\unitlength}{3947sp}%
\begingroup\makeatletter\ifx\SetFigFont\undefined%
\gdef\SetFigFont#1#2#3#4#5{%
  \reset@font\fontsize{#1}{#2pt}%
  \fontfamily{#3}\fontseries{#4}\fontshape{#5}%
  \selectfont}%
\fi\endgroup%
\begin{picture}(5276,5169)(406,-4616)
\put(421,-1200){\makebox(0,0)[lb]{\smash{{\SetFigFont{12}{14.4}{\familydefault}{\mddefault}{\updefault}{$A$}%
}}}}
\put(1350,-2900){\makebox(0,0)[lb]{\smash{{\SetFigFont{12}{14.4}{\familydefault}{\mddefault}{\updefault}{$B$}%
}}}}
\put(3200,-4141){\makebox(0,0)[lb]{\smash{{\SetFigFont{12}{14.4}{\familydefault}{\mddefault}{\updefault}{$C$}%
}}}}
\put(4500,-2986){\makebox(0,0)[lb]{\smash{{\SetFigFont{12}{14.4}{\familydefault}{\mddefault}{\updefault}{$D$}%
}}}}
\put(1500,0){\makebox(0,0)[lb]{\smash{{\SetFigFont{12}{14.4}{\familydefault}{\mddefault}{\updefault}{$F$}%
}}}}
\put(3900,-1329){\makebox(0,0)[lb]{\smash{{\SetFigFont{12}{14.4}{\familydefault}{\mddefault}{\updefault}{$E$}%
}}}}
\put(1239,-586){\makebox(0,0)[lb]{\smash{{\SetFigFont{12}{14.4}{\familydefault}{\mddefault}{\updefault}{$a$}%
}}}}
\put(4300,-646){\makebox(0,0)[lb]{\smash{{\SetFigFont{12}{14.4}{\familydefault}{\mddefault}{\updefault}{$\ell$}%
}}}}
\put(3025,131){\makebox(0,0)[lb]{\smash{{\SetFigFont{12}{14.4}{\familydefault}{\mddefault}{\updefault}{$d_1$}%
}}}}
\put(3400,-283){\makebox(0,0)[lb]{\smash{{\SetFigFont{12}{14.4}{\familydefault}{\mddefault}{\updefault}{$d_2$}%
}}}}
\put(3500,-706){\makebox(0,0)[lb]{\smash{{\SetFigFont{12}{14.4}{\familydefault}{\mddefault}{\updefault}{$d_3$}%
}}}}
\put(2500,-4231){\makebox(0,0)[lb]{\smash{{\SetFigFont{12}{14.4}{\familydefault}{\mddefault}{\updefault}{$d'_1$}%
}}}}
\put(1800,-3700){\makebox(0,0)[lb]{\smash{{\SetFigFont{12}{14.4}{\familydefault}{\mddefault}{\updefault}{$d'_2$}%
}}}}
\end{picture}%
\caption{The quasi-hexagonal region $H_{6}(4,4,3;\  5,5)$}
\label{hexagonregion}
\end{figure}

The northeastern boundary is defined analogously, starting from the lattice point $F$ on the top diagonal that is $a$ unit square diagonals to the northeast of $A$ (i.e. $|AF|=a\sqrt{2}$). Take unit steps south or east so that all fundamental regions on the left are white, until a point $E$ on $\ell$ is reached, and continue with similar unit steps so that the fundamental regions on the left are black, until a point $D$ on the bottom diagonal is reached. The described path from $F$ to $D$ completes the boundary of $H_a(d_1,\dotsc,d_k;d'_1,\dotsc,d'_l)$. We are only interested in connected regions, so we assume in addition that \textit{the northeastern and the southwestern boundaries do not intersect each other.}

Our generalization of Problem 16 of \cite{Propp} consists of providing an explicit formula for the number of tilings of such regions. We need some additional definitions and terminology as below.

Call the fundamental regions inside $H_a(d_1,\dotsc,d_k;d'_1,\dotsc,d'_l)$ \textit{cells}, and call them black or white according to the coloring described above.  Note that there are two kinds of cells, square and triangular. The latter in turn come in two orientations: they may point towards $\ell$ or away from $\ell$.  A cell is called \textit{regular} if it is a square cell or a triangular cell pointing away from $\ell$.

A \textit{row of cells} consists of all the triangular cells of a given color with bases resting on a fixed lattice diagonal, or  consists of all the square cells (of a given color) passed through by a fixed lattice diagonal.

Define the \textit{upper height} of our region to be the number of rows of regular black cells above $\ell$. The \textit{lower height} is the number of rows of white regular cells below $\ell$ (for example, the quasi-hexagon in Figure \ref{hexagonregion} has both the upper and lower heights equal to 6).

We are now ready to state our main result.
\begin{thm}\label{main}
 Let $a$ be a positive integer, and let $d_1,\dotsc,d_k,d'_1,\dotsc,d'_l$ be positive integers.

$(${\rm a}$)$. $\M(H_a(d_1,\dotsc,d_k;d'_1,\dotsc,d'_l))=0$ unless the upper and lower heights are equal and the bottom row of cells is black.

$(${\rm b}$)$. Suppose that the upper and lower heights of $H_a(d_1,\dotsc,d_k;d'_1,\dotsc,d'_l)$ are both equal to $h$, and that the bottom row of cells is black. Let $C$ be the number of regular black cells above $\ell$, and $C'$ the number of regular white cells below $\ell$.
Then for any positive integer $a$ we have
\begin{align}\label{maineq}
\M(H_a(d_1,\dotsc,d_k;d'_1,\dotsc,d'_l)) &= 2^{C+C'-h(2a+2m-2n-h+1)}\notag\\
& \times
\begin{cases}
\M(H_{h,a+m-n-h,h}), &\text{if $a+m-n\geq h$,}\\
0, & \text{if  $a+m-n < h$,}
\end{cases}
\end{align}
where  $m$ is the number of rows of black triangular cells above $\ell$ pointing away from $\ell$, and $n$ is the number of rows of black triangular cells above $\ell$ pointing towards $\ell$.
\end{thm}

\begin{rmk}
By MacMahon's formula (\ref{McMahon}) and equality (\ref{maineq}),  the number of tilings of the quasi-hexagon is equal to
\begin{equation}\label{MMreduce}
2^{C+C'-h(2a+2m-2n-h+1)}\prod^{h}_{r=1}\prod^{a+m-n-h}_{s=1}\prod^{h}_{t=1}\frac{r+s+t-1}{r+s+t-2},
\end{equation}
when $a+m-n\geq h$.
\end{rmk}

Note that the original problem corresponds to $d_1=\cdots=d_k=d'_1=\cdots=d'_l=3$ and $k=l$. The resulting region is then precisely the quasi-hexagon of sides $a$, $k$, $k$, $a$, $k$, $k$.  In particular, for $k=l=2$ and $d_1=d_2=d'_1=d'_2=3$, (\ref{maineq}) states that the number of tilings of the quasi-hexagon shown in Figure \ref{hexagonp} is equal to
$2^{8}\M(H_{4,1,4})=2^9\cdot5\cdot7$, thus explaining how this factorization (mentioned explicitly in \cite{Propp}) comes about.

\medskip
In the case of odd $d_i$'s and $d'_j$'s, our quasi-hexagons have the following nice facts: the bottom row of cells is always black, all the black cells above $\ell$ and all the white cells below $\ell$ are regular,  and the upper and lower heights are equal to $\frac{1}{2}\left(k+\sum_{i=1}^k d_i\right)$ and  $\frac{1}{2}\left( l+\sum_{j=1}^l d'_i\right)$, respectively. Therefore, we have a special form of Theorem \ref{maineq} as follows.

\begin{cor}\label{odd-main} Let $d_1,\dotsc,d_k$ and $d'_1,\dotsc,d'_l$ be odd positive integers. Assume that
\begin{equation}\label{balancing}
k+\sum_{i=1}^k d_i = l+\sum_{j=1}^l d'_i=2h.
\end{equation}
Let $C$ be the number of black cells above $\ell$, and $C'$ be the number of white cells below $\ell$. Then for any positive integer $a$ we have
\begin{align}\label{maineqodd}
\M(H_a(d_1,\dotsc,d_k;d'_1,\dotsc,d'_l)) &=2^{C+C'-h(2a+2k-h+1)} \notag\\
&\times
\begin{cases}
 \M(H_{h,a+k-h,h}),& \text{ if $a+k\geq h$,} \\
0,& \text{ if  $a+k < h$.}
\end{cases}
\end{align}
\end{cor}

Note that when $d_1=\cdots=d_k=d'_1=\cdots=d'_l=3$ in the original quasi-hexagons, the balancing condition (\ref{balancing}) is equivalent to $k=l$.

\medskip

As mentioned in \cite{Propp}, we do not have a simple product formula for the number of tilings of the quasi-hexagon of  Problem 16 in \cite{Propp} when $a$, $b$, $c$  are all distinct. However, we show that our graph transformations can still be used to turn the dual graph of the quasi-hexagon into a conceptually simple graph consisting of a honeycomb part and an Aztec rectangle part, and therefore the number of perfect matchings of the former is equal to a power of $2$ times the number of perfect matchings of the latter (see Theorems \ref{asymodd} and \ref{asymmetric1}).

\begin{figure}\centering
\begin{picture}(0,0)%
\includegraphics{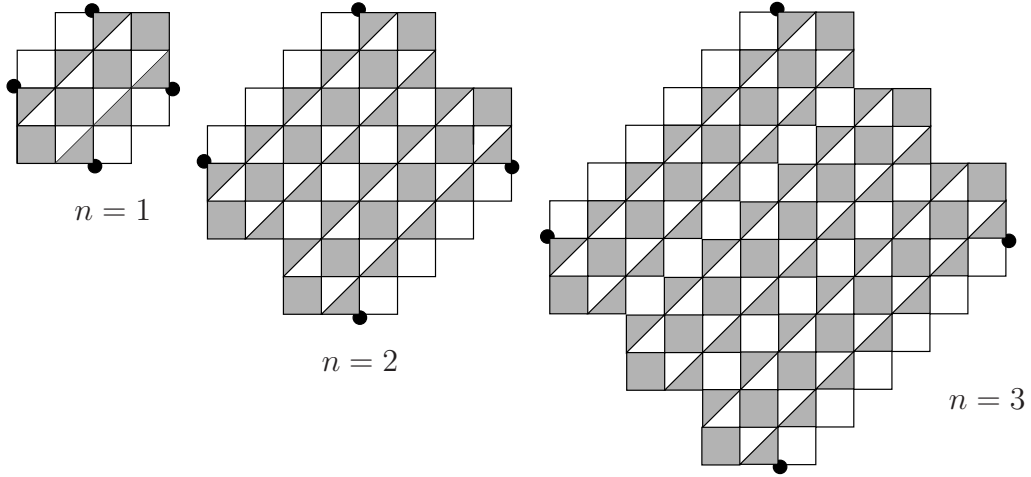}%
\end{picture}%
\setlength{\unitlength}{3947sp}%
\begingroup\makeatletter\ifx\SetFigFont\undefined%
\gdef\SetFigFont#1#2#3#4#5{%
  \reset@font\fontsize{#1}{#2pt}%
  \fontfamily{#3}\fontseries{#4}\fontshape{#5}%
  \selectfont}%
\fi\endgroup%
\begin{picture}(6283,2966)(402,-2292)
\put(2363,-1641){\makebox(0,0)[lb]{\smash{{\SetFigFont{12}{14.4}{\familydefault}{\mddefault}{\updefault}{$n=2$}%
}}}}
\put(6260,-1878){\makebox(0,0)[lb]{\smash{{\SetFigFont{12}{14.4}{\familydefault}{\mddefault}{\updefault}{$n=3$}%
}}}}
\put(828,-696){\makebox(0,0)[lb]{\smash{{\SetFigFont{12}{14.4}{\familydefault}{\mddefault}{\updefault}{$n=1$}%
}}}}
\end{picture}%
\caption{The Douglas' regions of order $n=1$, $n=2$ and $n=3$.}
\label{douglas}
\end{figure}

We also extend a theorem of Douglas \cite{Doug} concerning a particular family of regions on the lattice obtained from $\mathbb{Z}^2$ by drawing in every  \textit{second} diagonal (see Figure \ref{douglas}), to the case of arbitrary distances between the diagonals. Let $D_a(d_1,\dotsc,d_k)$ be the portion of $H_a(d_1,\dotsc,d_k;d'_1,\dotsc,d'_l)$ that is above the diagonal $\ell$ (i.e. the top half of the generalized quasi-hexagon); we call it a \textit{generalized Douglas region}. The \textit{height} of
$D_a(d_1,\dotsc,d_k)$ is defined to be equal to the upper height of $H_a(d_1,\dotsc,d_k;d'_1,\dotsc,d'_l)$. A triangle pointing towards $\ell$ in $D_a(d_1,\dotsc,d_k)$ is called a \textit{down-pointing triangle}, and a triangle pointing away from $\ell$ in this region is called a \textit{up-pointing triangle}.

Our extension of Douglas' result is the following.

\begin{thm}\label{douggen} Let $a$ be a positive integer, and let $d_1,\dotsc,d_k$ be positive integers. Let $h$ be the height of $D_a(d_1,\dotsc,d_k)$, $m$ (resp., $n$) be the number of rows of up-pointing (resp., down-pointing) triangular black cells in $D_a(d_1,\dotsc,d_k)$, and let $C$ be the number of black cells that are not down-pointing triangles.

$(${\rm a}$)$. If the the bottom row of cells of $D_a(d_1,\dotsc,d_k)$ is black, or if the equality
\begin{equation}\label{balancing2}
h=a+m-n
\end{equation}
fails, then $\M(D_a(d_1,\dotsc,d_k))=0$.

$(${\rm b}$)$.  Suppose (\ref{balancing2}) holds and the bottom row of cells of $D_a(d_1,\dotsc,d_k)$ is white. Then we have
\begin{equation}\label{douggeneq}
\M(D_a(d_1,\dotsc,d_k))
=
2^{C-h(h+1)} \M(AD_h)\notag\\
=
2^{C-h(h+1)/2}.
\end{equation}
\end{thm}

Douglas' theorem corresponds to the special case when $k$ is odd and $d_1=d_k=1$, $d_2=\cdots=d_{k-1}=2$, when the balancing condition (\ref{balancing2}) requires $a=k-1$ (Strictly speaking, in Douglas' original definition the top and bottom rows of cells consist of squares rather than triangles; however, the dual graph is isomorphic to the dual graph of our region, so the two versions have the same number of tilings). Thus, one can view Theorem \ref{main} as a common generalization of Douglas' theorem \cite{Doug} and Propp's open problem on quasi-hexagons.

\section{Preliminaries}

A \textit{perfect matching} of a graph $G$ is a collection of edges such that each vertex of $G$ is incident to precisely one edge in the collection. The number of perfect matchings of $G$ is denoted by $\M(G)$. More generally, if the edges of $G$ have weights on them, $\M(G)$ denotes the sum of the weights of all perfect matchings of $G$, where the weight of a perfect matching is the product of the weights on its constituent edges.

Given a lattice in the plane, a (lattice) \textit{region} is a finite connected union of fundamental regions of that lattice. A \textit{tile} is the union of any two fundamental regions sharing an edge. A \textit{tiling} of the region $R$ is a covering of $R$ by tiles with no gaps or overlaps. The tilings of a region $R$ can be naturally identified with the perfect matchings of its dual graph (i.e., the graph whose vertices are the fundamental regions of $R$, and whose edges connect two fundamental regions precisely when they share an edge). In view of this, we denote the number of tilings of $R$ by $\M(R)$.

A \textit{forced edge} of a graph is an edge contained in every perfect matching of $G$. Assume that $G$ is a weighted graph with weight function $\wt$ on its edges, and $G'$ is obtained from $G$ by removing forced edges $e_1,\dotsc,e_k$, and removing the vertices incident to those edges. Then one clearly has
\begin{equation*}
\M(G)=\M(G')\prod_{i=1}^k\wt(e_i).
\end{equation*}
\textit{Hereafter, whenever we remove some forced edges, we remove also the vertices incident to them.}

The main results of this section are two graph transformation rules that change the number of perfect matchings in a simple predictable way (see Lemmas 3.4 and 3.5). In proving them we will employ three basic preliminary results stated below.

\begin{figure}\centering
\begin{picture}(0,0)%
\includegraphics{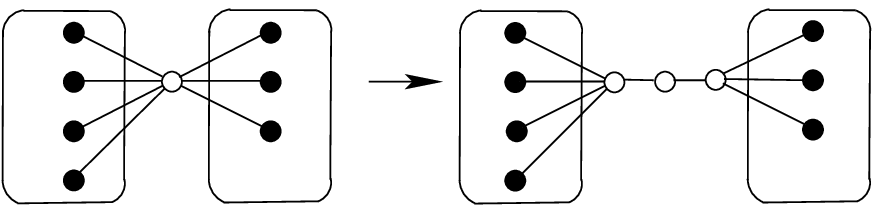}%
\end{picture}%
%
%
\setlength{\unitlength}{3947sp}%
\begingroup\makeatletter\ifx\SetFigFont\undefined%
\gdef\SetFigFont#1#2#3#4#5{%
  \reset@font\fontsize{#1}{#2pt}%
  \fontfamily{#3}\fontseries{#4}\fontshape{#5}%
  \selectfont}%
\fi\endgroup%
\begin{picture}(4188,1361)(593,-556)
\put(1336,591){\makebox(0,0)[lb]{\smash{{\SetFigFont{10}{12.0}{\familydefault}{\mddefault}{\updefault}{$v$}%
}}}}
\put(3549,621){\makebox(0,0)[lb]{\smash{{\SetFigFont{10}{12.0}{\familydefault}{\mddefault}{\updefault}{$v'$}%
}}}}
\put(3757, 89){\makebox(0,0)[lb]{\smash{{\SetFigFont{10}{12.0}{\familydefault}{\mddefault}{\updefault}{$x$}%
}}}}
\put(3999,621){\makebox(0,0)[lb]{\smash{{\SetFigFont{10}{12.0}{\familydefault}{\mddefault}{\updefault}{$v''$}%
}}}}
\put(820,-541){\makebox(0,0)[lb]{\smash{{\SetFigFont{10}{12.0}{\familydefault}{\mddefault}{\updefault}{$H$}%
}}}}
\put(1840,-535){\makebox(0,0)[lb]{\smash{{\SetFigFont{10}{12.0}{\familydefault}{\mddefault}{\updefault}{$K$}%
}}}}
\put(3031,-535){\makebox(0,0)[lb]{\smash{{\SetFigFont{10}{12.0}{\familydefault}{\mddefault}{\updefault}{$H$}%
}}}}
\put(4426,-484){\makebox(0,0)[lb]{\smash{{\SetFigFont{10}{12.0}{\familydefault}{\mddefault}{\updefault}{$K$}%
}}}}
\end{picture}%
\caption{}
\label{vertexsplitting}
\end{figure}
%
%

\begin{lem} [Vertex-Splitting Lemma]\label{VS}
 Let $G$ be a graph, $v$ be a vertex of it, and denote the set of neighbors of $v$ by $N(v)$.
  For any disjoint union $N(v)=H\cup K$, let $G'$ be the graph obtained from $G\setminus v$ by including three new vertices $v'$, $v''$ and $x$ so that $N(v')=H\cup \{x\}$, $N(v'')=K\cup\{x\}$, and $N(x)=\{v',v''\}$ (see Figure \ref{vertexsplitting}). Then $\M(G)=\M(G')$.
\end{lem}

\begin{lem}[Star Lemma]\label{star}
Let $G$ be a weighted graph, and let $v$ be a vertex of~$G$. Let $G'$ be the graph obtained from $G$ by multiplying the weights of all edges incident to $v$ by $t>0$. Then $\M(G')=t\M(G)$.
\end{lem}


Part (a) of the following result is a generalization (due to Propp) of the ``urban renewal" trick first observed by Kuperberg. Parts (b) and (c) are due to Ciucu (see Lemma 2.6 in [5]).

\begin{figure}\centering
\begin{picture}(0,0)%
\includegraphics{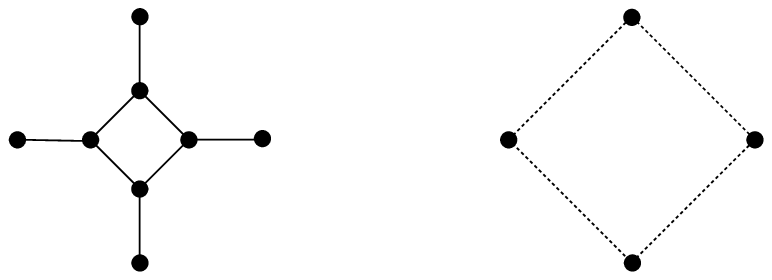}%
\end{picture}%
\setlength{\unitlength}{3947sp}%
\begingroup\makeatletter\ifx\SetFigFont\undefined%
\gdef\SetFigFont#1#2#3#4#5{%
  \reset@font\fontsize{#1}{#2pt}%
  \fontfamily{#3}\fontseries{#4}\fontshape{#5}%
  \selectfont}%
\fi\endgroup%
\begin{picture}(4054,1735)(340,-948)
\put(355,-156){\makebox(0,0)[lb]{\smash{{\SetFigFont{10}{12.0}{\familydefault}{\mddefault}{\updefault}{$A$}%
}}}}
\put(1064,-933){\makebox(0,0)[lb]{\smash{{\SetFigFont{10}{12.0}{\familydefault}{\mddefault}{\updefault}{$B$}%
}}}}
\put(1891,-106){\makebox(0,0)[lb]{\smash{{\SetFigFont{10}{12.0}{\familydefault}{\mddefault}{\updefault}{$C$}%
}}}}
\put(1182,603){\makebox(0,0)[lb]{\smash{{\SetFigFont{10}{12.0}{\familydefault}{\mddefault}{\updefault}{$D$}%
}}}}
\put(2717,-189){\makebox(0,0)[lb]{\smash{{\SetFigFont{10}{12.0}{\familydefault}{\mddefault}{\updefault}{$A$}%
}}}}
\put(3426,-933){\makebox(0,0)[lb]{\smash{{\SetFigFont{10}{12.0}{\familydefault}{\mddefault}{\updefault}{$B$}%
}}}}
\put(4253,-106){\makebox(0,0)[lb]{\smash{{\SetFigFont{10}{12.0}{\familydefault}{\mddefault}{\updefault}{$C$}%
}}}}
\put(3426,603){\makebox(0,0)[lb]{\smash{{\SetFigFont{10}{12.0}{\familydefault}{\mddefault}{\updefault}{$D$}%
}}}}
\put(904,-382){\makebox(0,0)[lb]{\smash{{\SetFigFont{10}{12.0}{\familydefault}{\mddefault}{\updefault}{$x$}%
}}}}
\put(1396,-388){\makebox(0,0)[lb]{\smash{{\SetFigFont{10}{12.0}{\familydefault}{\mddefault}{\updefault}{$y$}%
}}}}
\put(1418,130){\makebox(0,0)[lb]{\smash{{\SetFigFont{10}{12.0}{\familydefault}{\mddefault}{\updefault}{$z$}%
}}}}
\put(946,130){\makebox(0,0)[lb]{\smash{{\SetFigFont{10}{12.0}{\familydefault}{\mddefault}{\updefault}{$t$}%
}}}}
\put(2968,284){\makebox(0,0)[lb]{\smash{{\SetFigFont{10}{12.0}{\familydefault}{\mddefault}{\updefault}{$y/\Delta$}%
}}}}
\put(3934,311){\makebox(0,0)[lb]{\smash{{\SetFigFont{10}{12.0}{\familydefault}{\mddefault}{\updefault}{$x/\Delta$}%
}}}}
\put(3964,-544){\makebox(0,0)[lb]{\smash{{\SetFigFont{10}{12.0}{\familydefault}{\mddefault}{\updefault}{$t/\Delta$}%
}}}}
\put(2965,-526){\makebox(0,0)[lb]{\smash{{\SetFigFont{10}{12.0}{\familydefault}{\mddefault}{\updefault}{$z/\Delta$}%
}}}}
\put(2197,-817){\makebox(0,0)[lb]{\smash{{\SetFigFont{10}{12.0}{\familydefault}{\mddefault}{\updefault}{$\Delta= xz+yt$}%
}}}}
\end{picture}%
\caption{}
\label{spider1}
\end{figure}


\begin{figure}\centering
\resizebox{!}{3.7cm}{
\begin{picture}(0,0)%
\includegraphics{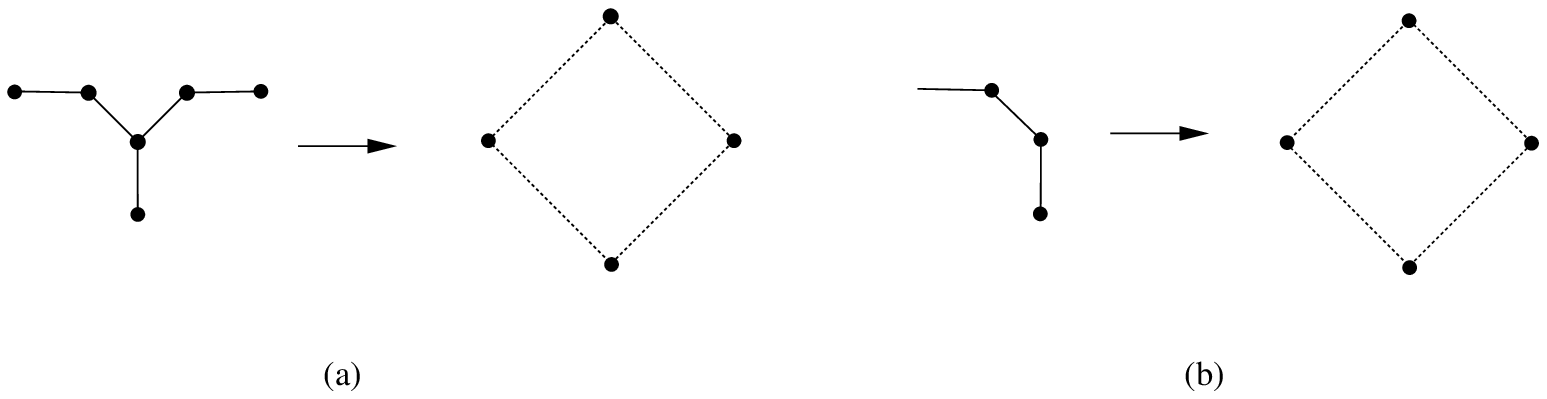}%
\end{picture}%
\setlength{\unitlength}{3947sp}%
\begingroup\makeatletter\ifx\SetFigFont\undefined%
\gdef\SetFigFont#1#2#3#4#5{%
  \reset@font\fontsize{#1}{#2pt}%
  \fontfamily{#3}\fontseries{#4}\fontshape{#5}%
  \selectfont}%
\fi\endgroup%
\begin{picture}(7605,2136)(222,-1352)
\put(237,-106){\makebox(0,0)[lb]{\smash{{\SetFigFont{10}{12.0}{\familydefault}{\mddefault}{\updefault}{$A$}%
}}}}
\put(828,-697){\makebox(0,0)[lb]{\smash{{\SetFigFont{10}{12.0}{\familydefault}{\mddefault}{\updefault}{$B$}%
}}}}
\put(1592,  0){\makebox(0,0)[lb]{\smash{{\SetFigFont{10}{12.0}{\familydefault}{\mddefault}{\updefault}{$C$}%
}}}}
\put(519,-249){\makebox(0,0)[lb]{\smash{{\SetFigFont{10}{12.0}{\familydefault}{\mddefault}{\updefault}{$x$}%
}}}}
\put(1164,-264){\makebox(0,0)[lb]{\smash{{\SetFigFont{10}{12.0}{\familydefault}{\mddefault}{\updefault}{$y$}%
}}}}
\put(2566,426){\makebox(0,0)[lb]{\smash{{\SetFigFont{10}{12.0}{\familydefault}{\mddefault}{\updefault}{$y/2$}%
}}}}
\put(3616,396){\makebox(0,0)[lb]{\smash{{\SetFigFont{10}{12.0}{\familydefault}{\mddefault}{\updefault}{$x/2$}%
}}}}
\put(2244,-714){\makebox(0,0)[lb]{\smash{{\SetFigFont{10}{12.0}{\rmdefault}{\mddefault}{\updefault}{$1/(2x)$}%
}}}}
\put(3646,-744){\makebox(0,0)[lb]{\smash{{\SetFigFont{10}{12.0}{\rmdefault}{\mddefault}{\updefault}{$1/(2y)$}%
}}}}
\put(2363,-129){\makebox(0,0)[lb]{\smash{{\SetFigFont{10}{12.0}{\familydefault}{\mddefault}{\updefault}{$A$}%
}}}}
\put(3131,600){\makebox(0,0)[lb]{\smash{{\SetFigFont{10}{12.0}{\familydefault}{\mddefault}{\updefault}{$D$}%
}}}}
\put(3830,-174){\makebox(0,0)[lb]{\smash{{\SetFigFont{10}{12.0}{\rmdefault}{\mddefault}{\updefault}{$C$}%
}}}}
\put(3119,-933){\makebox(0,0)[lb]{\smash{{\SetFigFont{10}{12.0}{\rmdefault}{\mddefault}{\updefault}{$B$}%
}}}}
\put(4489,249){\makebox(0,0)[lb]{\smash{{\SetFigFont{10}{12.0}{\rmdefault}{\mddefault}{\updefault}{$A$}%
}}}}
\put(5198,-814){\makebox(0,0)[lb]{\smash{{\SetFigFont{10}{12.0}{\rmdefault}{\mddefault}{\updefault}{$B$}%
}}}}
\put(6150,-115){\makebox(0,0)[lb]{\smash{{\SetFigFont{10}{12.0}{\rmdefault}{\mddefault}{\updefault}{$A$}%
}}}}
\put(6977,-942){\makebox(0,0)[lb]{\smash{{\SetFigFont{10}{12.0}{\rmdefault}{\mddefault}{\updefault}{$B$}%
}}}}
\put(7686,-115){\makebox(0,0)[lb]{\smash{{\SetFigFont{10}{12.0}{\rmdefault}{\mddefault}{\updefault}{$C$}%
}}}}
\put(6977,594){\makebox(0,0)[lb]{\smash{{\SetFigFont{10}{12.0}{\rmdefault}{\mddefault}{\updefault}{$D$}%
}}}}
\put(5251,179){\makebox(0,0)[lb]{\smash{{\SetFigFont{10}{12.0}{\rmdefault}{\mddefault}{\updefault}{$x$}%
}}}}
\put(6369,441){\makebox(0,0)[lb]{\smash{{\SetFigFont{10}{12.0}{\rmdefault}{\mddefault}{\updefault}{$1/2$}%
}}}}
\put(7501,-781){\makebox(0,0)[lb]{\smash{{\SetFigFont{10}{12.0}{\rmdefault}{\mddefault}{\updefault}{$1/2$}%
}}}}
\put(6046,-736){\makebox(0,0)[lb]{\smash{{\SetFigFont{10}{12.0}{\rmdefault}{\mddefault}{\updefault}{$1/(2x)$}%
}}}}
\put(7426,374){\makebox(0,0)[lb]{\smash{{\SetFigFont{10}{12.0}{\rmdefault}{\mddefault}{\updefault}{$x/2$}%
}}}}
\end{picture}}
\caption{}
\label{spider2}
\end{figure}

\begin{lem} [Spider Lemma]\label{spider}
(a) Let $G$ be a weighted graph containing the subgraph $K$ shown on the left in Figure \ref{spider1} (the labels indicate weights, unlabeled edges have weight 1). Suppose in addition that the four inner black vertices in the subgraph $K$, different from $A,B,C,D$, have no neighbors outside $K$. Let $G'$ be the graph obtained from $G$ by replacing $K$ by the graph $\overline{K}$ shown on right in Figure \ref{spider1}, where the dashed lines indicate new edges, weighted as shown. Then $\M(G)=(xz+yt)\M(G')$.

(b) Consider the above local replacement operation when $K$ and $\overline{K}$ are graphs shown in Figure \ref{spider2}(a) with the indicated weights (in particular, $K'$ has a new vertex $D$, that is incident only to $A$ and $C$). Then $\M(G)=2\M(G')$.

(c) The statement of part (b) is also true when $K$ and $\overline{K}$ are the graphs indicated in Figure \ref{spider2}(b) (in this case $G'$ has two new vertices $C$ and $D$, they are adjacent only to one another and to $B$ and $A$, respectively).
\end{lem}


\begin{figure}\centering%
\includegraphics[width=9cm]{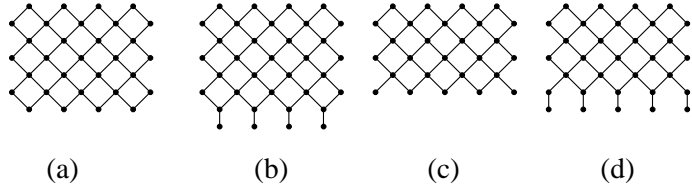}
\caption{Four graphs (a) $\AR_{3,4}$, (b) ${}_|\AR_{3,4}$, (c) $\AR_{3-\frac{1}{2},4}$, and (d) ${}_|\AR_{3-\frac{1}{2},4}$}
\label{ARvariants}
\end{figure}

The following four families of graphs will play a special role in our proofs. Consider a $(2m+1)\times(2n+1)$ rectangular chessboard $B$ and suppose the corners are black. The \textit{Aztec rectangle} $\AR_{m,n}$ is the graph whose vertices are the white squares and whose edges connect precisely those pairs of white squares that are diagonally adjacent (the case $m=3$, $n=4$ is shown in Figure \ref{ARvariants}(a)). If one removes the bottom row of the board $B$ and then applies the same procedure, the resulting graph is denoted by $\AR_{m-\frac12,n}$, and called a \textit{baseless Aztec rectangle} (see Figure \ref{ARvariants}(c) for an illustration)

The \textit{combed Aztec diamond }$_{|}\AR_{m,n}$ is the graph obtained from $\AR_{m,n}$ by appending a vertical edge to each of its bottom vertices (an example is shown in Figure \ref{ARvariants}(b)). The \textit{combed baseless Aztec rectangle} $_{|}\AR_{m-\frac12,n}$ is the graph obtained from $\AR_{m-\frac12,n}$ by appending a vertical edge to each of its bottom vertices (see  Figure \ref{ARvariants}(d) for an example).

The \textit{connected sum} $G\#G'$ of two disjoint graphs $G$ and $G'$ along the ordered sets of vertices $\{v_1,\dotsc,v_n\}\subset V(G)$ and $\{v'_1,\dotsc,v'_n\}\subset V(G')$ is the graph obtained from $G$ and $G'$ by identifying vertices $v_i$ and $v'_i$, for $i=1,\dotsc,n$.

\begin{figure}\centering
\resizebox{!}{8.5cm}{
\begin{picture}(0,0)%
\includegraphics{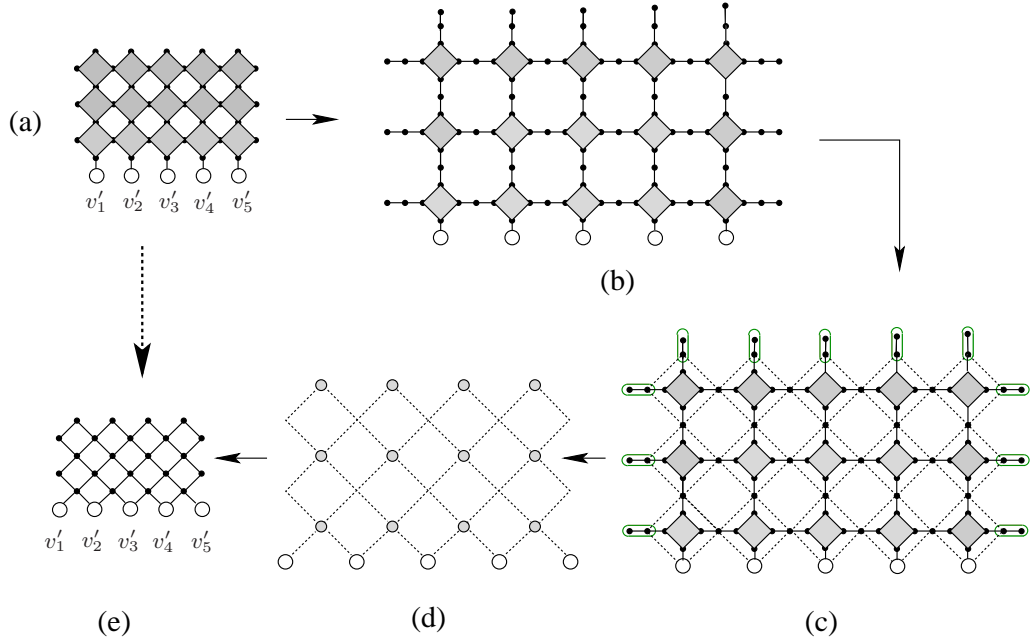}%
\end{picture}%
%
%
\setlength{\unitlength}{3947sp}%
\begingroup\makeatletter\ifx\SetFigFont\undefined%
\gdef\SetFigFont#1#2#3#4#5{%
  \reset@font\fontsize{#1}{#2pt}%
  \fontfamily{#3}\fontseries{#4}\fontshape{#5}%
  \selectfont}%
\fi\endgroup%
\begin{picture}(6787,4284)(282,-3699)
\put(807,-796){\makebox(0,0)[lb]{\smash{{\SetFigFont{9}{10.8}{\rmdefault}{\mddefault}{\updefault}{$v'_1$}%
}}}}
\put(1053,-793){\makebox(0,0)[lb]{\smash{{\SetFigFont{9}{10.8}{\rmdefault}{\mddefault}{\updefault}{$v'_2$}%
}}}}
\put(1293,-796){\makebox(0,0)[lb]{\smash{{\SetFigFont{9}{10.8}{\rmdefault}{\mddefault}{\updefault}{$v'_3$}%
}}}}
\put(1524,-802){\makebox(0,0)[lb]{\smash{{\SetFigFont{9}{10.8}{\rmdefault}{\mddefault}{\updefault}{$v'_4$}%
}}}}
\put(1779,-799){\makebox(0,0)[lb]{\smash{{\SetFigFont{9}{10.8}{\rmdefault}{\mddefault}{\updefault}{$v'_5$}%
}}}}
\put(534,-3065){\makebox(0,0)[lb]{\smash{{\SetFigFont{9}{10.8}{\rmdefault}{\mddefault}{\updefault}{$v'_1$}%
}}}}
\put(780,-3062){\makebox(0,0)[lb]{\smash{{\SetFigFont{9}{10.8}{\rmdefault}{\mddefault}{\updefault}{$v'_2$}%
}}}}
\put(1020,-3065){\makebox(0,0)[lb]{\smash{{\SetFigFont{9}{10.8}{\rmdefault}{\mddefault}{\updefault}{$v'_3$}%
}}}}
\put(1251,-3071){\makebox(0,0)[lb]{\smash{{\SetFigFont{9}{10.8}{\rmdefault}{\mddefault}{\updefault}{$v'_4$}%
}}}}
\put(1506,-3068){\makebox(0,0)[lb]{\smash{{\SetFigFont{9}{10.8}{\rmdefault}{\mddefault}{\updefault}{$v'_5$}%
}}}}
\end{picture}}
\caption{Transformation $T_1$. The dotted edges have weight $1/2$.}
\label{LemmaT1}
\end{figure}

\begin{lem}[Transformation $T_1$]\label{T1}

 Let $G$ be a graph and let  $\{v_1,\dotsc,v_q\}$ be an ordered subset of its vertices. Then
\begin{equation}\label{T1eq1}
\M\left({}_|\AR_{p,q}\#G\right)=2^p\M\left(\AR_{p-\frac12,q-1}\#G\right),
\end{equation}
where the connected sum acts on $G$ along $\{v_1,\dotsc,v_q\}$, and on ${}_|\AR_{p,q}$ and $\AR_{p-\frac12,q}$ along their bottom $q$ vertices (ordered from left to right).

\end{lem}

For $p=3,q=5$, the above transformation is illustrated in Figures \ref{LemmaT1}(a) and (e).
\begin{proof}
Let $G_1$ be the graph obtained from ${}_|AR_{p,q}\#G$ by applying the Vertex-splitting Lemma at all vertices of ${}_|AR_{p,q}$ that are not $v'_1,\dotsc,v'_q$ and not adjacent to any of them (Figures \ref{LemmaT1}(a) and (b) illustrate this for the case \\$p=3$ and $q=5$). Apply the Spider Lemma around all $pq$ shaded diamond in the graph $G_1$, and remove the $2p+q$ edges adjacent to a vertex of degree 1, which are forced edges (see Figure \ref{LemmaT1}(c), the forced edges are the circled ones). The resulting graph is isomorphic to $\AR^{\frac{1}{2}}_{p-\frac12,q-1}\#G$, where $\AR^{\frac{1}{2}}_{p-\frac12,q-1}$ is the graph obtained from $\AR_{p-\frac12,q-1}$ by changing all the weights of edges to $1/2$ (see Figure \ref{LemmaT1}(d); dotted edges have weight $1/2$). By Lemmas \ref{VS} and \ref{spider}, we have
\begin{equation}\label{T1eq2}
\M\left({}_|\AR_{p,q}\#G\right)=\M(G_1)=2^{pq}\M\left(\AR^{\frac{1}{2}}_{p-\frac12,q-1}\#G\right).
\end{equation}
Applying the Star Lemma with factor $t=2$ to all $p(q-1)$ shaded vertices of the graph $\AR^{\frac{1}{2}}_{p-\frac12,q-1}$ (see Figure \ref{LemmaT1}(d)), the graph $\AR^{\frac{1}{2}}_{p-\frac12,q-1}\#G$ is turned into  $\AR_{p-\frac12,q-1}\#G$. By (\ref{T1eq2}) and Lemma \ref{star}, we obtain
\begin{align}\label{T1eq3}
\M\left({}_|\AR_{p,q}\#G\right)&=2^{pq}\M\left(\AR^{\frac{1}{2}}_{p-\frac12,q-1}\#G\right)\notag\\
&=2^{pq}2^{-p(q-1)}\M\left(\AR_{p-\frac12,q-1}\#G\right),
\end{align}
which proves (\ref{T1eq1}).
\end{proof}




\begin{figure}\centering%
\resizebox{!}{9.5cm}{
\begin{picture}(0,0)%
\includegraphics{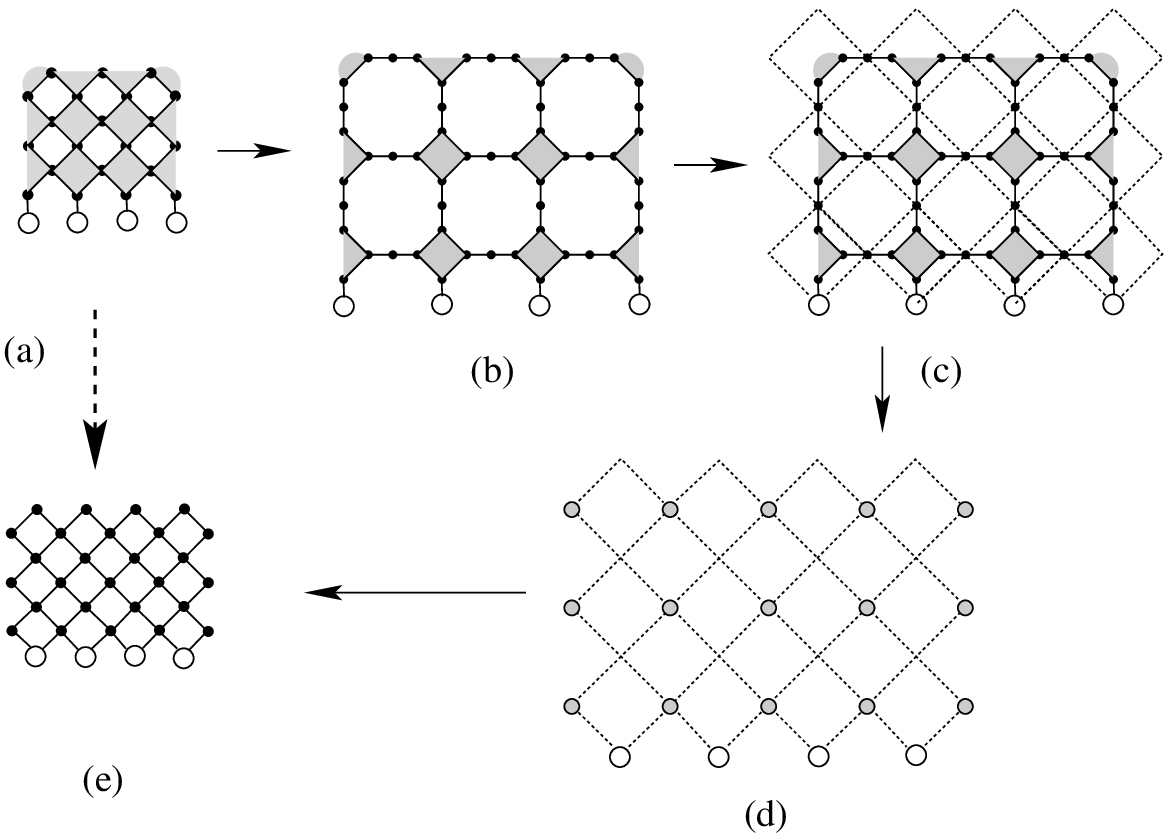}%
\end{picture}%
%
%
\setlength{\unitlength}{3947sp}%
\begingroup\makeatletter\ifx\SetFigFont\undefined%
\gdef\SetFigFont#1#2#3#4#5{%
  \reset@font\fontsize{#1}{#2pt}%
  \fontfamily{#3}\fontseries{#4}\fontshape{#5}%
  \selectfont}%
\fi\endgroup%
\begin{picture}(5593,3998)(489,-3710)
\put(554,-998){\makebox(0,0)[lb]{\smash{{\SetFigFont{9}{10.8}{\rmdefault}{\mddefault}{\updefault}{$v'_1$}%
}}}}
\put(797,-1004){\makebox(0,0)[lb]{\smash{{\SetFigFont{9}{10.8}{\rmdefault}{\mddefault}{\updefault}{$v'_2$}%
}}}}
\put(1034,-1001){\makebox(0,0)[lb]{\smash{{\SetFigFont{9}{10.8}{\rmdefault}{\mddefault}{\updefault}{$v'_3$}%
}}}}
\put(1268,-1007){\makebox(0,0)[lb]{\smash{{\SetFigFont{9}{10.8}{\rmdefault}{\mddefault}{\updefault}{$v'_4$}%
}}}}
\put(616,-3088){\makebox(0,0)[lb]{\smash{{\SetFigFont{9}{10.8}{\rmdefault}{\mddefault}{\updefault}{$v'_1$}%
}}}}
\put(859,-3094){\makebox(0,0)[lb]{\smash{{\SetFigFont{9}{10.8}{\rmdefault}{\mddefault}{\updefault}{$v'_2$}%
}}}}
\put(1096,-3091){\makebox(0,0)[lb]{\smash{{\SetFigFont{9}{10.8}{\rmdefault}{\mddefault}{\updefault}{$v'_3$}%
}}}}
\put(1330,-3097){\makebox(0,0)[lb]{\smash{{\SetFigFont{9}{10.8}{\rmdefault}{\mddefault}{\updefault}{$v'_4$}%
}}}}
\end{picture}}
\caption{Transformation $T_2$. The dotted edges have weight $1/2$.}
\label{LemmaT2}
\end{figure}

\begin{lem}[Transformation $T_2$]\label{T2}
 Let $G$ be a graph and let  $\{v_1,\dotsc,v_{q+1}\}$ be an ordered subset of its vertices. Then
\begin{equation}\label{T2eq1}
\M\left({}_|\AR_{p-\frac12,q}\#G\right)=2^{-p}\M(\AR_{p,q+1}\#G),
\end{equation}
where the connected sum acts on $G$ along $\{v_1,\dotsc,v_{q+1}\}$, and on ${}_|\AR_{p-\frac12,q}$ and $\AR_{p,q+1}$ along their bottom $q+1$ vertices ordered from left to right (see Figures \ref{LemmaT2}(a) and (e), for $p=3,q=3$).
\end{lem}

One can prove Lemma \ref{T2} similarly to Lemma \ref{T1} by using Vertex-splitting Lemma, Spider Lemma and Star Lemma  to transform the graph on the left hand side of (\ref{T2eq1}) into the graph on the right hand side. This transforming process is illustrated by Figure \ref{LemmaT2}, for $p=3,q=3$.

\medskip

An \textit{induced subgraph} of a graph $G$ is a graph whose vertex set is a subset $U$ of the vertex set of $G$, and whose edge set consists of all edges of $G$ with endpoints in $U$.  For a graph $G$, we denote by $V(G)$ its vertex set.

\begin{lem}[Graph Splitting Lemma]\label{graphsplitting}
Let $G$ be a bipartite graph, and let $V_1$ and $V_2$ be the two vertex classes. Let $H$ be an induced subgraph of $G$.

(a) Assume that $H$ satisfies the following  two conditions.
\begin{enumerate}
\item[(i)] The \textit{separating condition}: there are no edges of $G$ connecting a vertex in $V(H)\cap V_1$ and a vertex in $V(G-H)$,

\item[(ii)] The\textit{ balancing condition}: $|V(H)\cap V_1|=|V(H)\cap V_2|$.
\end{enumerate}
Then
\begin{equation}\label{GS}
\M(G)=\M(H)\, \M(G-H).
\end{equation}

(b) If $H$ satisfies the separating condition and but $|V(H)\cap V_1|>|V(H)\cap V_2|$, then $\M(G)=0$.
\end{lem}
\begin{proof}
Color all vertices of $V_1$ white, and all vertices of $V_2$ black. To prove part (a) of the lemma, it suffices to show that that none of the perfect matchings of $G$ contain any edge connecting a vertex of $H$ to a vertex of $G-H$. Since there are no edges of $G$ connecting a white vertex of $H$ to a vertex of $G-H$ (by the separating condition), the previous assertion follows by the balancing condition.

Part (b) is perfectly analogous, the only difference is that $H$ has no tilings when $|V(H)\cap V_1|>|V(H)\cap V_2|.$
\end{proof}

\section{Proof of Theorem \ref{douggen}}

The proof of Theorem 2.3 --- and as we will see in the next section, also the proof of Theorem 2.2 --- is based on a common generalization of the transformations $T_1$ and $T_2$ of Lemmas \ref{T1} and \ref{T2}, presented below in Proposition \ref{composite}.

The diagonals divide $D_a(d_1,\dotsc,d_k)$ into $k$ parts called \textit{layers}. Recall that $m$ and $n$ are the numbers of rows of triangular black cells of $D_a(d_1,\dotsc,d_k)$ that point up and down, respectively. As one moves down from top, each row of black up-pointing triangles causes the length of the rows of black cells in the next layer to go up by one, and each row of black down-pointing triangles causes the length of the rows of black cells in that layer to be one less than in the previous layer. This implies that the bottom row of cells of $D_a(d_1,\dotsc,d_k)$ consists of $a+m-n$ cells (note that these are always triangular and up-pointing, but they may be either black or white). The \textit{width of a layer} is one less than the length of its rows of blacks cells.

Let $G_a(d_1,\dotsc,d_k)$ be the dual graph of the region $D_a(d_1,\dotsc,d_k)$.

\begin{figure}\centering
\includegraphics[width=12cm]{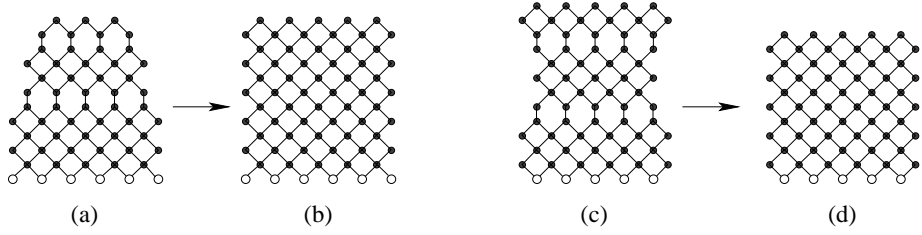}
\caption{Four graphs (a) $G_{3}(1,3,5)$, (b) $\AR_{6-1/2,5}$, (c) $G_{5}(2,4,4,)$, and (d) $\AR_{5,5}$.}
\label{mass}
\end{figure}

\begin{prop}[Composite Transformation]\label{composite} Let $a$ and $d_1,\dotsc,d_k$ be positive integers. Recall that $C$ is the number of black cells of $D_a(d_1,\dotsc,d_k)$ that are either squares or up-pointing triangles, and $h$ is the number of rows of black cells consisting of  squares or up-pointing triangles. Let $G$ be a graph, and consider an ordered subset of its vertices $\{v_1,\dotsc,v_{a+m-n}\}$.

(a) If the bottom row of cells in $D_a(d_1,\dotsc,d_k)$ is white, then
\begin{equation}\label{composeq1}
\M(G_a(d_1,\dotsc,d_k)\#G)=2^{C-h(a+m-n+1)}\M(\AR_{h,a+m-n}\#G),
\end{equation}
where the connected sum acts on $G$ along $\{v_1,\dotsc,v_{a+m-n}\}$, and on
the other two summands along their bottom $a+m-n$ vertices (ordered from left to right). See Figures \ref{mass}(c) and (d) for an example, the white circles indicate the vertices $\{v_1,$ $\dotsc,v_{a+m-n}\}$.

(b) If the bottom row of cells in $D_a(d_1,\dotsc,d_k)$ is black, then
\begin{equation}\label{composeq2}
\M(G_a(d_1,\dotsc,d_k)\#G)=2^{C-h(a+m-n)}\M\left(\AR_{h-\frac{1}{2},a+m-n-1}\#G\right),
\end{equation}
where the connected sum acts on $G$ along $\{v_1,\dotsc,v_{a+m-n}\}$, and on
the other two summands along their bottom $a+m-n$ vertices (ordered from left to right). See Figures \ref{mass}(a) and (b) for an example.

\end{prop}

\begin{figure}\centering
\includegraphics[width=12cm]{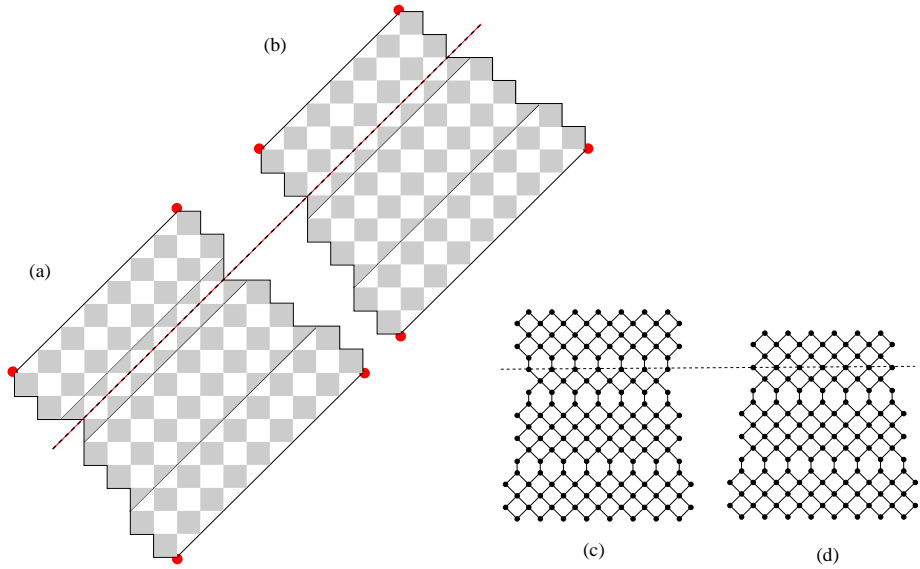}
\caption{Two regions (a) $D_{7}(4,2,5,4)$ and (b) $D_{6}(5,5,4)$; and their dual graphs (c) $G_{7}(4,2,5,4)$ and (d) $G_{6}(5,5,4)$.}
\label{compare1}
\end{figure}

\begin{proof} (a) We prove (\ref{composeq1}) by induction on $k$. For $k=1$, the dual graph $G_a(d_1,\dotsc,d_k)$ of the region is just the Aztec rectangle $\AR_{h,a}$, and $m=n=0$. In particular, the number $C$ of black cells is equal to $h(a+1)$, and (\ref{composeq1}) holds.

For the induction step, suppose $k>1$ and assume (\ref{composeq1}) holds for all regions with strictly less than $k$ \textit{layers} (the layers of $G_a(d_1,\dotsc,d_k)$ correspond to the layers of $D_a(d_1,\dotsc,d_k)$, which are the portions of the latter between consecutive diagonals). Note that the top layer of $G_a(d_1,\dotsc,d_k)$ is either an Aztec rectangle or a baseless Aztec rectangle, based on whether $d_1$ is even or odd, respectively. We treat the two cases separately.

Suppose that the top layer of $G_a(d_1,\dotsc,d_k)$ is the Aztec rectangle $\AR_{p,a}$, i.e. $d_1$ is even and $p=\frac{d_1}{2}$. Applying the transformation $T_1$ of Lemma \ref{T1} with respect to this layer, we obtain
\begin{equation}\label{composeq3}
\M(G_a(d_1,\dotsc,d_k)\#G)=2^p\M(K\#G),
\end{equation}
where $K$ is the graph obtained from $G_a(d_1,\dotsc,d_k)$ by replacing its top part, isomorphic to the combed Aztec rectangle ${}_|\AR_{p,a}$, by the baseless Aztec rectangle $\AR_{p-\frac12,a-1}$. Note that the graph $K$ is also the dual of a generalized Douglas region $D'$ with white bottom row of cells, but the number of layers of $K$ is $k-1$ (see Figures \ref{compare1}(c) and (d) for an example; the subgraph above the dotted line in graph (c) is replaced by the subgraph above that line in graph (d)). In precise, $D'=D_{a-1}(d_1+d_2-1,d_3,\dotsc,d_{k})$. Note that $a-1$ is the width of the second layer in $G_{a}(d_1,\dotsc,d_k)$, so $a-1>0$.

Therefore, by the induction hypothesis, we have
\begin{equation}\label{composeq4}
\M(K\#G)=2^{C'-h'(a'+m'-n'+1)}\M(\AR_{h',a'+m'-n'}\#G),
\end{equation}
where the primed symbols refer to the region $D'$ and denote quantities corresponding to their unprimed counterparts in $D_a(d_1,\dotsc,d_k)$.

It is clear by definition of $D'$ that $a'=a-1$. Two regions $D_a(d_1,\dotsc,d_k)$ and $D'$ are different only in their top parts, and their bottom parts are the same (see Figures \ref{compare1}(a) and (b); the parts below the dotted line in two regions are the same). The number of rows of black up-pointing triangles is the same in $D_a(d_1,\dotsc,d_k)$ and $D'$, but the latter has one fewer rows of black down-pointing triangles (because the former region has such a row just below its top layer, and that is no longer present in $D'$). Thus $m'=m$, and $n'=n-1$. Moreover, the numbers of rows of black squares in $D_{a}(d_1,\dotsc,d_k)$ and $D'$ are equal, so we get $h=h'$.
Finally, one easily sees that $C-C'=p(a+1)-pa=p$. Therefore, we obtain
\begin{equation}\label{composeq5}
C-h(a+m-n+1)-\{C'-h'(a'+m'-n'+1)\}=p.
\end{equation}
Then, by (\ref{composeq3})--(\ref{composeq5}), we have
\begin{equation*}\label{composeq6}
\M(G_a(d_1,\dotsc,d_k)\#G)=2^{C-h(a+m-n+1)}\M(\AR_{h,a+m-n}\#G),
\end{equation*}
which completes the induction step in this case.


\begin{figure}\centering
\includegraphics[width=12cm]{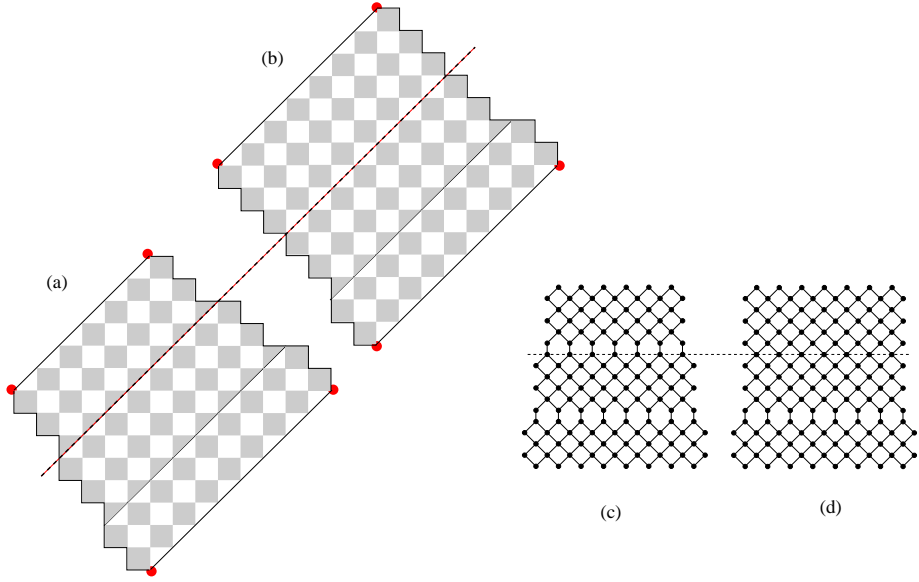}
\caption{Two regions (a) $D_{6}(5,5,4)$ and (b) $D_{7}(11,4)$; and their dual graphs (c) $G_{6}(5,5,4)$ and (d) $G_{7}(11,4)$.}
\label{compare2}
\end{figure}

Suppose now that the top layer of $G_a(d_1,\dotsc,d_k)$ is isomorphic to the baseless Aztec rectangle $\AR_{p-\frac12,a}$, i.e. $d_1$ is odd and $p=\frac{d_1+1}{2}$. Applying this time the transformation $T_2$ of Lemma \ref{T2} with respect to this layer, we obtain that
\begin{equation}\label{composeq7}
\M(G_a(d_1,\dotsc,d_k)\#G)=2^{-p}\M(K\#G),
\end{equation}
where $K$ is now the graph obtained from $G_a(d_1,\dotsc,d_k)$ by replacing its top part, isomorphic to the combed baseless Aztec rectangle ${}_|\AR_{p-\frac12,a}$, by the Aztec rectangle $\AR_{p,a+1}$ (see Figures \ref{compare2}(c) and (d) for an example). Again, $K$ is the dual of a generalized Douglas region $D':=D_{a+1}(d_1+d_2+1,d_3,\dotsc,d_{k})$ of the same kind as $D_a(d_1,\dotsc,d_k)$ (i.e., with a white bottom row of cells), but with fewer layers. 
Therefore, by the induction hypothesis we  have again the equality (\ref{composeq4}).

Again, it is clear that $a'=a+1$, and two regions $D_a(d_1,\dotsc,d_k)$ and $D'$ are different only in their top layers (see Figures \ref{compare2}(a) and (b) for an example). 
One can verify that $m'=m-1$ and $n'=n$ in this case. Moreover, the number of rows of black squares is now one fewer in $D_{a}(d_1,\dotsc,d_k)$ than in $D'$, thus we still have $h=h'$.
Finally, one readily sees that $C-C'$ equals now $p(a+1)-p(a+2)=-p$. This yields
\begin{equation}\label{composeq9}
C-h(a+m-n+1)-\{C'-h'(a'+m'-n'+1)\}=-p,
\end{equation}
which in turn, by (\ref{composeq7}), (\ref{composeq4}) and (\ref{composeq9}), implies
\begin{equation*}\label{composeq10}
\M(G_a(d_1,\dotsc,d_k)\#G)=2^{C-h(a+m-n+1)}\M(\AR_{h,a+m-n}\#G).
\end{equation*}
This completes the induction step in the second case, and thus also the proof of (\ref{composeq1}). Part (b) is proved by a perfectly analogous argument.
\end{proof}


The enumeration of tilings of the region $D_a(d_1,\dotsc,d_k)$ follows easily using the above result.
\medskip
\begin{proof}[Proof of Theorem \ref{douggen}.]
 Apply Proposition \ref{composite} with $G$ chosen to be the graph consisting of $a+m-n$ vertices and no edges. If the bottom row of cells in $D_a(d_1,\dotsc,d_k)$ is black, then (\ref{composeq2}) implies that
\begin{equation}\label{dougeq1}
\M(G_a(d_1,\dotsc,d_k))=2^{C-h(a+m-n)}\M\left(\AR_{h-\frac12,a+m-n-1}\right).
\end{equation}
Since a baseless Aztec rectangle has no perfect matchings (because it is impossible to cover all its bottom vertices by disjoint edges), it follows that in this case $\M(G_a(d_1,\dotsc,d_k))=0$.

Assume therefore that the bottom row of cells in $D_a(d_1,\dotsc,d_k)$ is white. Then with $G$ chosen like above, (\ref{composeq1}) implies that
\begin{equation}\label{dougeq2}
\M(G_a(d_1,\dotsc,d_k))=2^{C-h(a+m-n+1)}\M(\AR_{h,a+m-n}).
\end{equation}
It follows thus that unless $h=a+m-n$, one again has $\M(G_a(d_1,\dotsc,d_k))=0.$ This proves part (a) of the theorem. Part (b) follows directly from (\ref{dougeq2}) and the Aztec diamond theorem (\ref{diamond}). \end{proof}

\section{Proof of Theorem \ref{main}.}

Our proof uses the composite transformations of Proposition \ref{composite} to transform the dual graph of a generalized quasi-hexagon into a honeycomb graph whose perfect matchings are enumerated by MacMahon's formula (\ref{McMahon}).

\begin{figure}\centering

\begin{picture}(0,0)%
\includegraphics{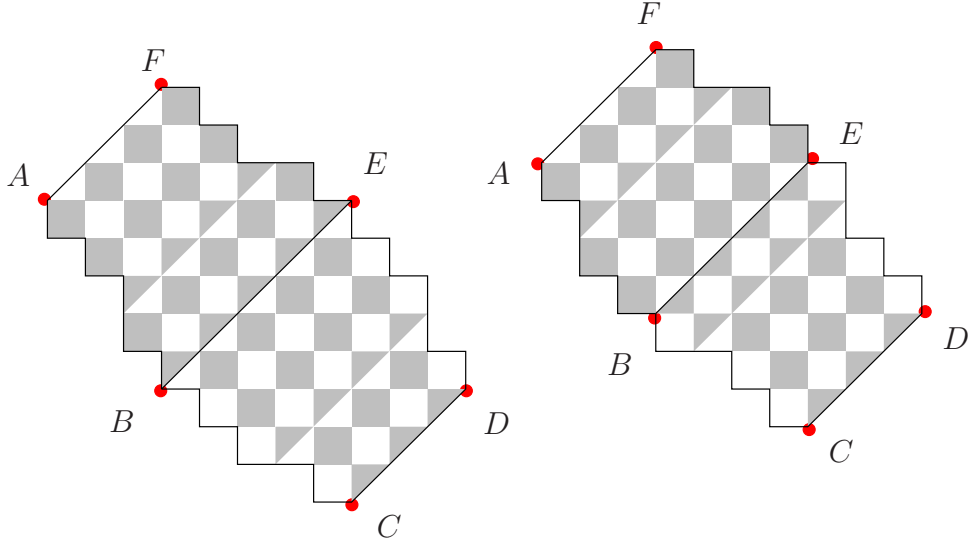}%
\end{picture}%
\setlength{\unitlength}{3947sp}%
\begingroup\makeatletter\ifx\SetFigFont\undefined%
\gdef\SetFigFont#1#2#3#4#5{%
  \reset@font\fontsize{#1}{#2pt}%
  \fontfamily{#3}\fontseries{#4}\fontshape{#5}%
  \selectfont}%
\fi\endgroup%
\begin{picture}(5994,3441)(913,-4120)
\put(928,-1906){\makebox(0,0)[lb]{\smash{{\SetFigFont{12}{14.4}{\familydefault}{\mddefault}{\updefault}{$A$}%
}}}}
\put(1564,-3448){\makebox(0,0)[lb]{\smash{{\SetFigFont{12}{14.4}{\familydefault}{\mddefault}{\updefault}{$B$}%
}}}}
\put(3229,-4105){\makebox(0,0)[lb]{\smash{{\SetFigFont{12}{14.4}{\familydefault}{\mddefault}{\updefault}{$C$}%
}}}}
\put(3889,-3460){\makebox(0,0)[lb]{\smash{{\SetFigFont{12}{14.4}{\familydefault}{\mddefault}{\updefault}{$D$}%
}}}}
\put(3142,-1831){\makebox(0,0)[lb]{\smash{{\SetFigFont{12}{14.4}{\familydefault}{\mddefault}{\updefault}{$E$}%
}}}}
\put(1762,-1183){\makebox(0,0)[lb]{\smash{{\SetFigFont{12}{14.4}{\familydefault}{\mddefault}{\updefault}{$F$}%
}}}}
\put(3913,-1894){\makebox(0,0)[lb]{\smash{{\SetFigFont{12}{14.4}{\familydefault}{\mddefault}{\updefault}{$A$}%
}}}}
\put(4660,-3088){\makebox(0,0)[lb]{\smash{{\SetFigFont{12}{14.4}{\familydefault}{\mddefault}{\updefault}{$B$}%
}}}}
\put(6037,-3634){\makebox(0,0)[lb]{\smash{{\SetFigFont{12}{14.4}{\familydefault}{\mddefault}{\updefault}{$C$}%
}}}}
\put(6742,-2923){\makebox(0,0)[lb]{\smash{{\SetFigFont{12}{14.4}{\familydefault}{\mddefault}{\updefault}{$D$}%
}}}}
\put(6097,-1642){\makebox(0,0)[lb]{\smash{{\SetFigFont{12}{14.4}{\familydefault}{\mddefault}{\updefault}{$E$}%
}}}}
\put(4843,-889){\makebox(0,0)[lb]{\smash{{\SetFigFont{12}{14.4}{\familydefault}{\mddefault}{\updefault}{$F$}%
}}}}
\end{picture}
\caption{Two regions $H_3(5,3;\ 3,5)$ (left) and $H_3(3,4;\ 5,2)$ (right)}
\label{conhex}
\end{figure}

\begin{figure}\centering
\includegraphics[width=8cm]{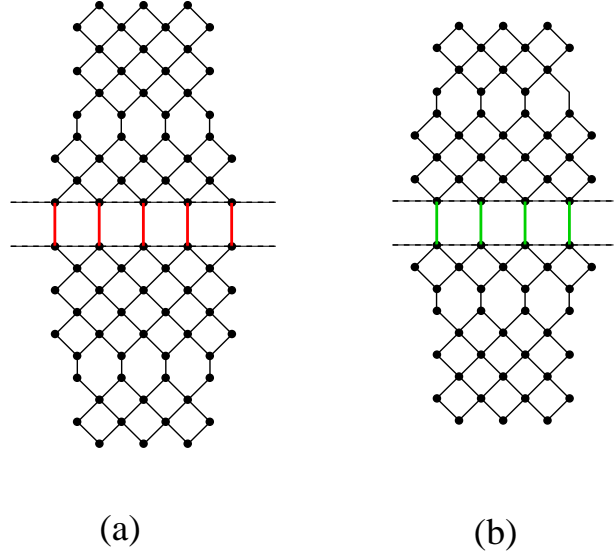}
\caption{The graphs (a) $G_{3}(5,3;\ 3,5)$ and (b) $G_{3}(3,4;\ 5,3)$.}
\label{conhex2}
\end{figure}

\medskip
\begin{proof} [Proof of Theorem \ref{main}.]
First,  we consider the case where the bottom row of cells is white. Then we have the same situation when showing that a baseless Aztec rectangle does not have a perfect matching. It is also impossible to cover the bottom vertices in the dual graph by disjoint edges. Therefore the dual graph has no perfect matchings, i.e. the region has no tilings. Hereafter, we assume that the bottom row of cells of our region is black.


There are two cases to distinguish, depending on whether the row of triangular cells just above $\ell$ and with bases resting on $\ell$ is black or white (these two cases are illustrated in Figure \ref{conhex}; the corresponding dual graphs are shown in Figure \ref{conhex2}).

Suppose the row of cells just above $\ell$ consists of black triangles (note that this is the case for our motivating problem; see Figure \ref{hexagonp}). Denote by
\[G_a(d_1,\dotsc,d_k;\ d'_1,\dotsc,d'_l)\]
 the dual graph of the generalized quasi-hexagon $H_a(d_1,\dotsc,d_k;d'_1,\dotsc,d'_l).$

Apply the composite transformation in Proposition \ref{composite}(b) separately to the portions of $G_a(d_1,\dotsc,d_k;\ d'_1,\dotsc,d'_l)$ corresponding to the parts above and below $\ell$ of the region $H$ (which we will call the \textit{upper} and \textit{lower parts} of the dual graph; illustrated in Figure \ref{conhex2} as the parts above and below the dotted lines). We obtain that
\begin{subequations}\begin{align}\label{maineq1}
&\M(G_a(d_1,\dotsc,d_k;\ d'_1,\dotsc,d'_l))=2^{C-h(a+m-n)}2^{C'-h'(a'+m'-n')}\notag\\
&\qquad\times\M\left({}_|\AR_{h-\frac12,a+m-n-1}\#\AR_{h'-\frac12,a'+m'-n'-1}\right)\\
&\qquad=2^{C-h(a+m-n)}2^{C'-h'(a'+m'-n')}\M(G_{a+m-n-1}(2h-1; 2h'-1)),
\end{align}
\end{subequations}
where the connected sum acts along the bottom vertices of the two modified Aztec rectangles (ordered from left to right), and $h'$, $a'$, $m'$ and $n'$ are quantities having the same significance for the lower part of $G_a(d_1,\dotsc,d_k;\ d'_1,\dotsc,d'_l)$ as their unprimed correspondents have for its upper part (except that --- due to the geometry of $H_a(d_1,\dotsc,d_k;d'_1,\dotsc,d'_l)$ and the resulting coloring of its fundamental regions --- they count \textit{white} cells). The transformation involved in (\ref{maineq1}) is illustrated in Figures \ref{symmetricfig}(a) and~(b).

\begin{figure}\centering
\includegraphics[width=10cm]{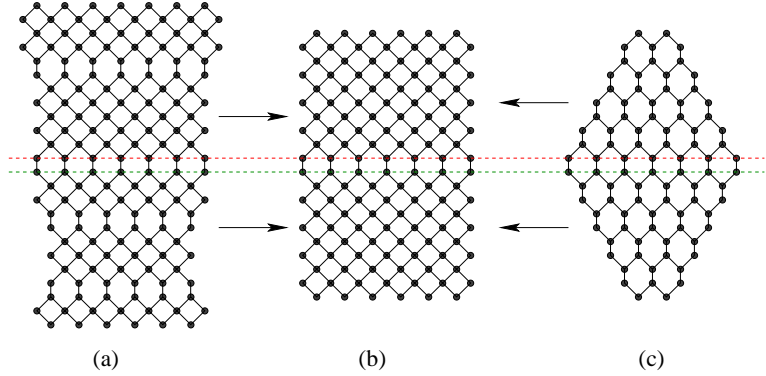}
\caption{The transforming process for generalized quasi-hexagons in the first case. The upper parts are above the dotted lines, the lower parts are below the dotted lines.}
\label{symmetricfig}
\end{figure}

As we have seen in the second paragraph of Section 4, $a+m-n$ is the number of triangular cells above $\ell$ with bases resting on $\ell$. By the same argument, $a'+m'-n'$ is the number of triangular cells below $\ell$, with bases resting on $\ell$. These two numbers are equal by construction, and therefore we have
\begin{equation}\label{maineq2}
a+m-n=a'+m'-n'.
\end{equation}
Clearly, a necessary condition for the graph on the right hand side of (\ref{maineq1}) (which is a bipartite graph) to have a perfect matching is that the numbers of its black and white vertices are the same. Denote for brevity $q:=a+m-n$. Then its upper part has $(q-1)h$ white and $qh$ black vertices, while its lower part has $(q-1)h'$ black and $qh'$ white vertices. The balancing condition is then
\begin{equation}\label{maineq2'}
(q-1)h+qh'=qh+(q-1)h',
\end{equation}
which is equivalent to $h=h'$. This proves in particular the rest of part (a) of the theorem.

To prove part (b), we consider first the case when $a+m-n>h$. The key idea here is to realize that the composite transformation of Proposition \ref{composite}(b) can also be used to turn a honeycomb graph into the graph on the right hand side of (\ref{maineq1}). More precisely, consider the honeycomb graph of sides $a+m-n-h,h,h,a+m-n-h,h,h$ (clockwise from top), 
illustrated in Figure \ref{symmetricfig}(c). Note that this is the dual graph of a generalized quasi-hexagon, in which all the inter-diagonal distances $d_i$'s and $d'_j$'s are equal to one  (i.e. the lozenge hexagonal region of the same sides)! Therefore equality (\ref{maineq1}) applies in this case, and the question is what is exactly its form in this particular case (see Figures \ref{symmetricfig} (c) and (b)). By symmetry, the primed and unprimed quantities of $C$-, $h$-, $a$-, $m$- and $n$-type are equal. Furthermore, one readily sees that the number of black regular cells above $\ell$ is
\begin{equation*}
h(a+m-n)-\binom{h}{2}.
\end{equation*}
The upper height is equal to $h$ and so is the $m$-parameter, while the $n$-parameter is 0. The $a$-parameter equals $a+m-n-h$. It follows that the particular form (\ref{maineq1}) takes is
\begin{equation}\label{maineq3}
\M(H_{a+m-n-h,h,h})=2^{-h(h-1)}\M\left({}_|\AR_{h-\frac12,a+m-n-1}\#\AR_{h-\frac12,a+m-n-1}\right).
\end{equation}

Equations (\ref{maineq1}) (in which $a'+m'-n'=a+m-n$ by (\ref{maineq2}), and $h'=h$ by assumption) and (\ref{maineq3}) imply the statement of part (b) of the theorem, provided $a+m-n-h>0$.


\begin{figure}\centering%
\begin{picture}(0,0)%
\includegraphics{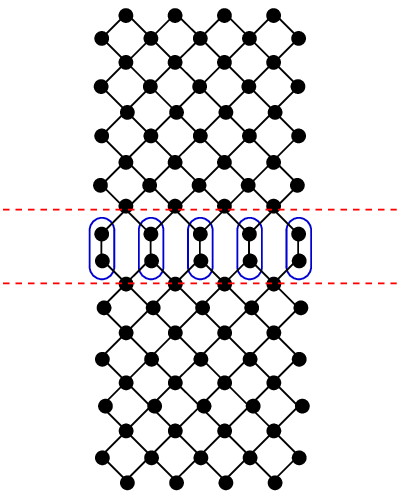}%
\end{picture}%
\setlength{\unitlength}{3947sp}%
\begingroup\makeatletter\ifx\SetFigFont\undefined%
\gdef\SetFigFont#1#2#3#4#5{%
  \reset@font\fontsize{#1}{#2pt}%
  \fontfamily{#3}\fontseries{#4}\fontshape{#5}%
  \selectfont}%
\fi\endgroup%
\begin{picture}(1917,2321)(3057,-2614)
\put(3000,-815){\makebox(0,0)[lb]{\smash{{\SetFigFont{12}{14.4}{\rmdefault}{\mddefault}{\updefault}{$\AR^{(1)}$}%
}}}}
\put(3000,-2114){\makebox(0,0)[lb]{\smash{{\SetFigFont{12}{14.4}{\rmdefault}{\mddefault}{\updefault}{$\AR^{(2)}$}%
}}}}
\end{picture}
\caption{}
\label{Ben1s}
\end{figure}

Note that if $a+m-n=h$, then the lozenge hexagonal region of side-lengths $a+m-n-h(=0),h,h,a+m-n-h(=0),h,h$ is \textit{not} a member of our family of quasi-hexagons anymore (in the definition of a quasi-hexagon, we always assume that $A$ and $F$ are two distinct lattice vertices); and if $a+m-n<h$, then the lozenge hexagonal region with side-lengths above does not exist (two side-lengths are negative). This means that the method used to prove (\ref{maineq}) when $a+m-n>h$ does \textit{not} work for the case when $a+m-n\leq h$.

\medskip
To complete the proof of (\ref{maineq}), we need to show that the graph
\[\overline{G}:={}_|\AR_{h-\frac12,a+m-n-1}\#\AR_{h-\frac12,a+m-n-1}\]
 has $2^{h(h-1)}$ perfect matchings when $a+m-n=h$, and that it has no perfect matchings when $a+m-n<h$. The top $2h-1$ rows of its vertices induce a subgraph isomorphic to $\AR_{h-1,a+m-n-1}$, and the bottom $2h-1$ rows of its vertices induce also a subgraph isomorphic to $\AR_{h-1,a+m-n-1}$. Denote by $\AR^{(1)}$ and $\AR^{(2)}$ the two Aztec rectangles (illustrated by the subgraph above and the subgraph below two dotted lines in Figure \ref{Ben1s}).  If $a+m-n<h$, then $\AR^{(1)}$ satisfies the conditions in  Graph-Splitting Lemma \ref{graphsplitting}(b) as an induced subgraph of $\overline{G}$, and thus $M(\overline{G})=0$.

 Assume that $a+m-n=h$, then the two Aztec rectangles $\AR^{(1)}$ and $\AR^{(2)}$ become two Aztec diamonds of order $(h-1)$. The Aztec rectangle $\AR^{(1)}$ satisfies the conditions in Graph-splitting Lemma \ref{graphsplitting}(a) as an induced subgraph of $\overline{G}$, and the rectangle $AR^{(2)}$ satisfies also the conditions of this lemma as an induced subgraph of $\overline{G}-\AR^{(1)}$. Therefore
\begin{subequations}
\begin{align}\label{power2}
\M(\overline{G})&=\M(\AR^{(1)})\M(\overline{G}-\AR^{(1)})\\
&=\M(\AR^{(1)})\M(\AR^{(2)})\M(\overline{G}-\AR^{(1)}-\AR^{(2)})\\
&=2^{h(h-1)/2}2^{h(h-1)/2},
\end{align}
\end{subequations}
where graph $\overline{G}-\AR^{(1)}-\AR^{(2)}$ consists of $h$ disjoint vertical edges (see the circled edges in Figure \ref{Ben1s}), so it has exactly one perfect matching. This finishes the proof of part (b) of the theorem in the case when the triangular cells above $\ell$ with bases resting on it are black.

\begin{figure}\centering
\includegraphics[width=11cm]{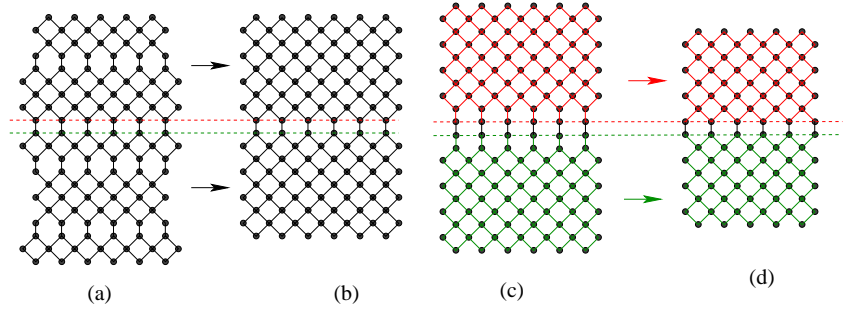}
\caption{The transforming process for generalized quasi-hexagons in the second case.}
\label{symmetricfig2}
\end{figure}

Suppose now that the latter are white. Apply in this case the composite transformation in Proposition \ref{composite}(a) to the upper and lower parts of the graph $G_a(d_1,\dotsc,d_k;d'_1,\dotsc,d'_l)$. We obtain that
\begin{subequations}
\begin{align}\label{maineq4}
&\M(G_a(d_1,\dotsc,d_k;d'_1,\dotsc,d'_l))\notag\\
&
\qquad=2^{C-h(a+m-n+1)}2^{C'-h'(a'+m'-n'+1)}\M\left({}_|\AR_{h,a+m-n}\#\AR_{h',a'+m'-n'}\right)\\
&
\qquad=2^{C-h(a+m-n+1)}2^{C'-h'(a'+m'-n'+1)}\M(G_{a+m-n}(2h;2h')),
\end{align}
\end{subequations}
where the connected sum acts along the bottom vertices of the two Aztec rectangles (ordered from left to right), and $h'$, $a'$, $m'$ and $n'$ are quantities having the same significance for the lower part of $G_a(d_1,\dotsc,d_k;d'_1,\dotsc,d'_l)$ as their unprimed correspondents have for its upper part (except that, as in the first case, they now count white cells). The transformation involved in (\ref{maineq4}) is illustrated in Figures \ref{symmetricfig2}(a) and~(b). Apply the Vertex-splitting Lemma \ref{VS} to all bottom vertices of ${}_|\AR_{h,a+m-n}$ in the graph on the right hand side of (\ref{maineq4}) (see Figures \ref{symmetricfig2}(b) and~(c)). Next, apply the transformation $T_1$ of Lemma \ref{T1} to the top and bottom of the resulting graph, the latter is transformed into the graph on the right hand side of (\ref{maineq1}) (see Figures \ref{symmetricfig2}(c) and (d)). Then both parts (a) and (b) of the theorem are reduced to the case treated above.~\end{proof}


\section{Asymmetric quasi-hexagons}

The generalized quasi-hexagons we treated so far had the special property that the point on the southwestern boundary where its defining property ``black on left'' changed to ``white on left'' (i.e. vertex $B$) was on the same southwest-to-northeast lattice diagonal as the analogous point on the northeastern boundary (i.e. vertex $E$). Therefore, the perpendicular bisector of segment $AF$ is the symmetric axis of the quasi-hexagons. These include as a special case Propp's quasi-hexagons of side-lengths $a$, $b$, $b$, $a$, $b$, $b$ (clockwise from top), and will be called from now on symmetric (generalized) quasi-hexagons.

\begin{figure}\centering
\resizebox{!}{10cm}{
\begin{picture}(0,0)%
\includegraphics{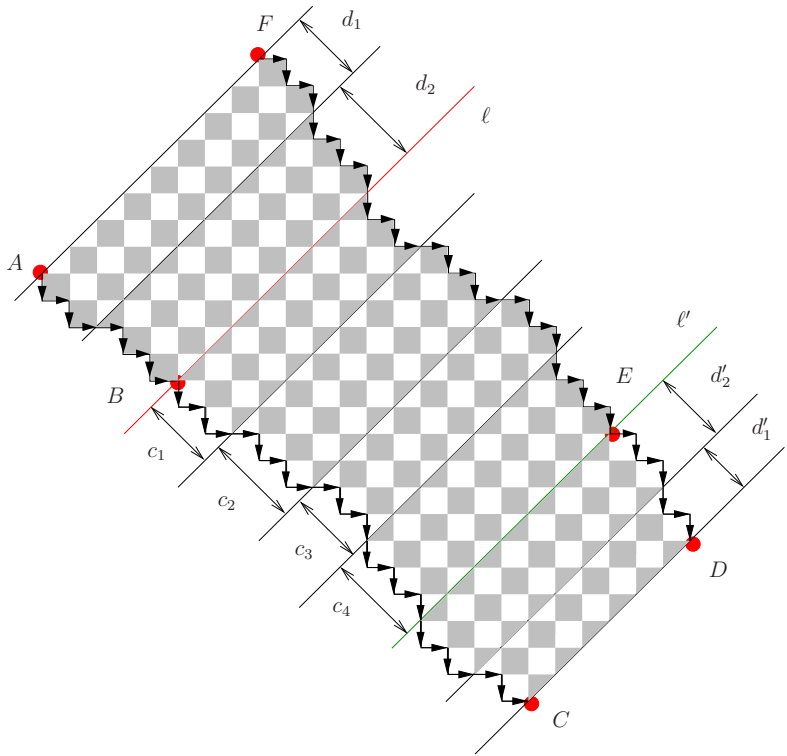}%
\end{picture}%
\setlength{\unitlength}{3947sp}%
\begingroup\makeatletter\ifx\SetFigFont\undefined%
\gdef\SetFigFont#1#2#3#4#5{%
  \reset@font\fontsize{#1}{#2pt}%
  \fontfamily{#3}\fontseries{#4}\fontshape{#5}%
  \selectfont}%
\fi\endgroup%
\begin{picture}(6845,6664)(609,-7199)
\put(4786,-1606){\makebox(0,0)[lb]{\smash{{\SetFigFont{12}{14.4}{\rmdefault}{\mddefault}{\updefault}{$\ell$}%
}}}}
\put(6497,-3413){\makebox(0,0)[lb]{\smash{{\SetFigFont{12}{14.4}{\rmdefault}{\mddefault}{\updefault}{$\ell'$}%
}}}}
\put(624,-2904){\makebox(0,0)[lb]{\smash{{\SetFigFont{12}{14.4}{\rmdefault}{\mddefault}{\updefault}{$A$}%
}}}}
\put(1501,-4081){\makebox(0,0)[lb]{\smash{{\SetFigFont{12}{14.4}{\rmdefault}{\mddefault}{\updefault}{$B$}%
}}}}
\put(5409,-6946){\makebox(0,0)[lb]{\smash{{\SetFigFont{12}{14.4}{\rmdefault}{\mddefault}{\updefault}{$C$}%
}}}}
\put(2814,-796){\makebox(0,0)[lb]{\smash{{\SetFigFont{12}{14.4}{\rmdefault}{\mddefault}{\updefault}{$F$}%
}}}}
\put(5964,-3909){\makebox(0,0)[lb]{\smash{{\SetFigFont{12}{14.4}{\rmdefault}{\mddefault}{\updefault}{$E$}%
}}}}
\put(6781,-5619){\makebox(0,0)[lb]{\smash{{\SetFigFont{12}{14.4}{\rmdefault}{\mddefault}{\updefault}{$D$}%
}}}}
\put(1861,-4561){\makebox(0,0)[lb]{\smash{{\SetFigFont{12}{14.4}{\rmdefault}{\mddefault}{\updefault}{$c_1$}%
}}}}
\put(2476,-4996){\makebox(0,0)[lb]{\smash{{\SetFigFont{12}{14.4}{\rmdefault}{\mddefault}{\updefault}{$c_2$}%
}}}}
\put(3151,-5431){\makebox(0,0)[lb]{\smash{{\SetFigFont{12}{14.4}{\rmdefault}{\mddefault}{\updefault}{$c_3$}%
}}}}
\put(3481,-5934){\makebox(0,0)[lb]{\smash{{\SetFigFont{12}{14.4}{\rmdefault}{\mddefault}{\updefault}{$c_4$}%
}}}}
\put(6796,-3894){\makebox(0,0)[lb]{\smash{{\SetFigFont{12}{14.4}{\rmdefault}{\mddefault}{\updefault}{$d'_2$}%
}}}}
\put(4209,-1321){\makebox(0,0)[lb]{\smash{{\SetFigFont{12}{14.4}{\rmdefault}{\mddefault}{\updefault}{$d_2$}%
}}}}
\put(3564,-751){\makebox(0,0)[lb]{\smash{{\SetFigFont{12}{14.4}{\rmdefault}{\mddefault}{\updefault}{$d_1$}%
}}}}
\put(7156,-4359){\makebox(0,0)[lb]{\smash{{\SetFigFont{12}{14.4}{\rmdefault}{\mddefault}{\updefault}{$d'_1$}%
}}}}
\end{picture}}
\caption{The asymmetric quasi-hexagon $H_{8}(4,5;\ 4,5,4,5;\ 3,4)$.}
\label{asymmetrichex}
\end{figure}

In order to include the general case of Propp's quasi-hexagons when the side-lengths are $a$, $b$, $c$, $a$, $b$, $c$, we extend the definition of our generalized quasi-hexagons as follows. Besides $\ell$, we consider a second distinguished lattice diagonal $\ell'$ below it, and draw in also $t-1$ additional diagonals in between them, so that the distances between successive ones (starting from top) are $c_1,\dotsc,c_t$. As in the symmetric case, we consider $k$ diagonals above $\ell$ at successive distances (starting from top) $d_1,\dotsc,d_k$, and $l$ diagonals below $\ell'$, at successive distances (starting from the bottom) $d'_1,\dotsc,d'_l$. Define our generalized quasi-hexagon as in the symmetric case, with the one change that on the northeastern boundary we make the switch from the rule ``white on left'' to ``black on left'' on the diagonal $\ell'$, rather than on $\ell$. Denote the resulting region by $H_a(d_1,\dotsc,d_k;c_1,\dotsc,c_t;d'_1,\dotsc,d'_l)$; an illustrative example is shown in Figure \ref{asymmetrichex}.

As mentioned in Section 2, there is no simple product formula for the number of tilings of asymmetric quasi-hexagons. However, it turns out that we can construct a conceptually simple graph so that the number of tilings of an asymmetric quasi-hexagon is equal to a certain power of $2$ times the number of perfect matchings of this graph. We present this construction next.

Roughly speaking, the graph we need is a connected sum of a baseless Aztec rectangle and the bottom half of a honeycomb graph. To be precise, let $a$, $b$, $c$ be positive integers with $b\geq c$, and let $B_{a,b,c}$ be the portion of the honeycomb graph of sides $a$, $2b-c$, $c$, $a$, $2b-c$, $c$ (clockwise from top) that is below or on the horizontal through its center. Let $d$ and $e$ be positive integers, and consider the baseless Aztec rectangle $\AR_{d-\frac12,e}$. The latter has $e+1$ vertices on the bottom; label them from left to right by $1,2,\dotsc,e+1$. The graph $B_{a,b,c}$ has $a+c$ vertices at the top; label them from left to right by $1,2,\dotsc,a+c$.


\begin{figure}\centering%
\begin{picture}(0,0)%
\includegraphics{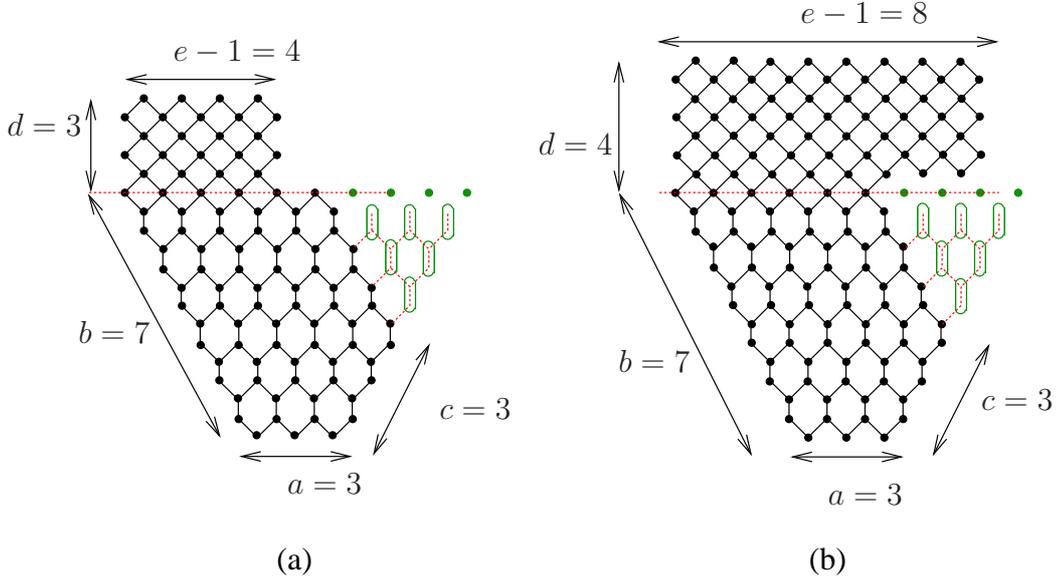}%
\end{picture}
\setlength{\unitlength}{3947sp}%
\begingroup\makeatletter\ifx\SetFigFont\undefined%
\gdef\SetFigFont#1#2#3#4#5{%
  \reset@font\fontsize{#1}{#2pt}%
  \fontfamily{#3}\fontseries{#4}\fontshape{#5}%
  \selectfont}%
\fi\endgroup%
\begin{picture}(6349,3716)(222,-3488)
\put(1891,-2941){\makebox(0,0)[lb]{\smash{{\SetFigFont{12}{14.4}{\rmdefault}{\mddefault}{\updefault}{$a=3$}%
}}}}
\put(2836,-2468){\makebox(0,0)[lb]{\smash{{\SetFigFont{12}{14.4}{\rmdefault}{\mddefault}{\updefault}{$c=3$}%
}}}}
\put(592,-1996){\makebox(0,0)[lb]{\smash{{\SetFigFont{12}{14.4}{\rmdefault}{\mddefault}{\updefault}{$b=7$}%
}}}}
\put(150,-696){\makebox(0,0)[lb]{\smash{{\SetFigFont{12}{14.4}{\rmdefault}{\mddefault}{\updefault}{$d=3$}%
}}}}
\put(3948,-2178){\makebox(0,0)[lb]{\smash{{\SetFigFont{12}{14.4}{\rmdefault}{\mddefault}{\updefault}{$b=7$}%
}}}}
\put(5250,-3005){\makebox(0,0)[lb]{\smash{{\SetFigFont{12}{14.4}{\rmdefault}{\mddefault}{\updefault}{$a=3$}%
}}}}
\put(6203,-2414){\makebox(0,0)[lb]{\smash{{\SetFigFont{12}{14.4}{\rmdefault}{\mddefault}{\updefault}{$c=3$}%
}}}}
\put(3450,-814){\makebox(0,0)[lb]{\smash{{\SetFigFont{12}{14.4}{\rmdefault}{\mddefault}{\updefault}{$d=4$}%
}}}}
\put(1182,-224){\makebox(0,0)[lb]{\smash{{\SetFigFont{12}{14.4}{\rmdefault}{\mddefault}{\updefault}{$e-1=4$}%
}}}}
\put(5080, 12){\makebox(0,0)[lb]{\smash{{\SetFigFont{12}{14.4}{\rmdefault}{\mddefault}{\updefault}{$e-1=8$}%
}}}}
\end{picture}%
\caption{Two graphs of solid edges $\Gamma^{3,4}_{3,7,3}$ (left) and  $\Gamma^{4,8}_{3,7,3}$ (right).}
\label{Cfigure}
\end{figure}

If $a+c\geq e+1$, define the graph  $\Gamma_{a,b,c}^{d,e}$ by
\begin{equation}
\Gamma_{a,b,c}^{d,e}:=\AR_{d-\frac12,e}\#B_{a,b,c},
\end{equation}
where the connected sum identifies the vertices labeled $i$ in the two components, for $i=1,\dotsc,e+1$.

On the other hand, if $a+c< e+1$, define $\Gamma_{a,b,c}^{d,e}$ by
\begin{equation}
\Gamma_{a,b,c}^{d,e}:=\AR^*_{d-\frac12,e}\#B_{a,b,c},
\end{equation}
where $\AR^*_{d-\frac12,e}$ is the subgraph of $\AR_{d-\frac12,e}$ obtained by deleting its vertices with labels greater than $a+c$, and the connected sum identifies the vertices labeled $i$ in the two components, for $i=1,\dotsc,a+c$. Examples illustrating the two cases are shown in Figure~\ref{Cfigure}.

For the sake of simplicity, we detail below the case when all distances between successive diagonals are odd (recall that the original problem, Problem 16 in \cite{Propp}, corresponds to the special case when all these distances are equal to 3). The general case is addressed in Theorem~\ref{asymmetric1}.

\begin{thm}\label{asymodd}
 Let $d_1,\dotsc,d_k$, $d'_1,\dotsc,d'_l$ and $c_1,\dotsc,c_k$ be odd integers, and set
\[h_0:=\sum_{i=1}^t\frac{c_i-1}{2}, \qquad h:=\sum_{i=1}^k\frac{d_i+1}{2}, \qquad h':=\sum_{i=1}^l\frac{d'_i+1}{2}.\]
Let $H:=H_a(d_1,\dotsc,d_k;c_1,\dotsc,c_t;d'_1,\dotsc,d'_l)$ be an asymmetric quasi-hexagon, and let $C$ be the number of black cells above $\ell$, and $C'$ the number of white cells below $\ell'$.

$(${\rm a}$)$. If $h\neq h'$, then $\M(H)=0$.

$(${\rm b}$)$. Suppose that $h=h'$ and $a+k>h$.  Then for any positive integer $a$ we have
\begin{align}\label{asymoddeq}
\M(H) &=2^{C+C'-h(2a+2k-h+1)}\notag\\
&\times
\begin{cases}
\M(H_{h,a+k-h,h+t}), &\text{\rm if  $h_0=0$,}\\
2^{-h_0(h_0-1)/2}\M\left(\Gamma_{a+k-h,h_0+2h+t,h}^{h_0,h_0+h-1}\right),
&\text{\rm if  $h_0>0$.}
\end{cases}
\end{align}

$(${\rm c}$)$. Suppose that $h=h'$. If $a+k<h$, then $\M(H)=0$. If $a+k=h$, then $\M(H)=2^{C+C'-h(2a+2k-h+1)}$.

\end{thm}

In order to prove this result we will need some more graph transformations, which we present in the next two lemmas. Let $LR_{m,n}$ be the graph obtained from the Aztec rectangle $\AR_{m,n}$ by deleting its leftmost $m$ vertices, and let $RR_{m,n}$ be the graph obtained from $\AR_{m,n}$ by deleting its rightmost $m$ vertices.
\begin{figure}\centering
\includegraphics[width=14cm]{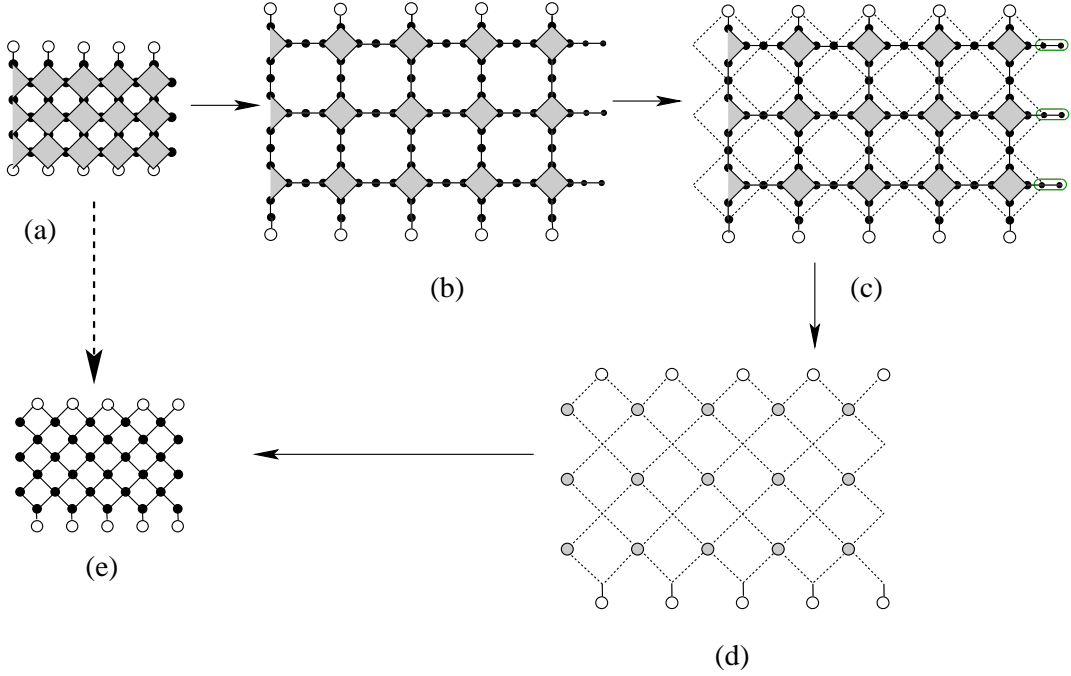}
\caption{Transformation $T_3$. The dotted edges have weight $1/2$.}
\label{LemmaT3}
\end{figure}

\begin{lem} [Transformation $T_3$]\label{T3} Let $G$ be a graph, and let $U$ be an ordered $2n$-element subset of its vertex set.

Let ${}^|LR_{m,n}$ be the graph obtained from the graph $LR_{m,n}$ by appending vertical edges to its top vertices, and ${}_|RR_{m,n}$ the graph obtained from $RR_{m,n}$ by appending vertical edges to its bottom vertices.
Then
\begin{equation}\label{T3eq1}
\M\left(G\#{}^|LR_{m,n}\right)=\M\left(G\#{}_|RR_{m,n}\right),
\end{equation}
where the connected sum acts along the top and bottom vertices of ${}^|LR_{m,n}$ and ${}_|RR_{m,n}$ (ordered from left to right and from top to bottom), and along the ordered set $U$ of vertices of $G$.
\end{lem}

The above transformation is illustrated by Figures \ref{LemmaT3}(a) and (e), for $m=3$ and $n=5$. One can prove Lemma \ref{T3} by following the argument in the proof of Lemma \ref{T1}, based on Figure \ref{LemmaT3}.

A special role will be played in our proof of Theorem \ref{asymodd}  by two families of L-shaped graphs, which we describe next.

Let $a,b,c,d$ be positive integers. The graph $L^{a,b}_{c,d}$ is obtained from the baseless Aztec rectangle $AR_{a-\frac12,b}$ and $LR_{c,d}$ as follows.

\begin{figure}\centering
\includegraphics[width=11cm]{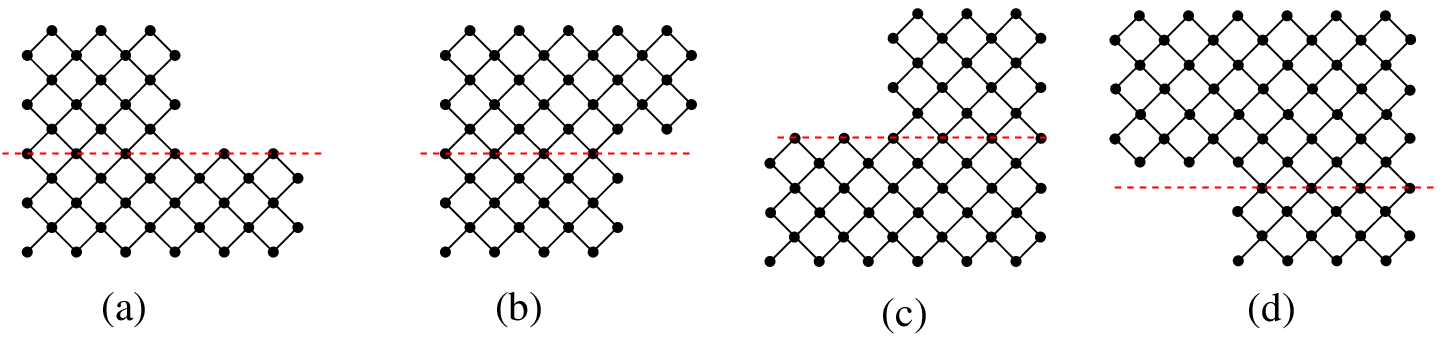}
\caption{Four L-shaped graphs (a) $L^{3,3}_{2,6}$, (b) $L^{3,5}_{2,4}$, (c) $\overline{L}^{3,3}_{3,6}$, and (c) $\overline{L}^{4,6}_{2,4}$.}
\label{Lgraph}
\end{figure}

Note that $AR_{a-\frac12,b}$ has $b+1$ vertices on the bottom, and $LR_{c,d}$ has $d$ vertices on top. If $b+1\leq d$, define
\begin{equation}\label{Lgrapheq1}
L^{a,b}_{c,d}:=AR_{a-\frac12,b}\#LR_{c,d},
\end{equation}
where the connected sum identifies the $i$th vertex on the bottom in $AR_{a-\frac12,b}$ with the $i$th vertex on top in $LR_{c,d}$ (both counted from left to right), for $i=1,\dotsc,b+1$ (the graph $L^{3,3}_{2,6}$ is shown in Figure \ref{Lgraph}(a)).

On the other hand, if $b+1>d$, define our graph by
\begin{equation}\label{Lgrapheq2}
L^{a,b}_{c,d}:=\AR^*_{a-\frac12,b}\#LR_{c,d},
\end{equation}
where $\AR^*_{a-\frac12,b}$ is the graph obtained from  $\AR_{a-\frac12,b}$ by deleting the rightmost $b+1-d$ of its bottom vertices, and the connected sum identifies the $i$th vertex on the bottom in $\AR^*_{a-\frac12,b}$ with the $i$th vertex on top in $LR_{c,d}$, for $i=1,$ $2,\dotsc,d$ ($L^{3,5}_{2,4}$ is pictured in Figure~\ref{Lgraph}(b)).

Our second family of graphs is the result of an analogous construction, in which the role of $LR_{m,n}$ is replaced by the graph $TLR_{m,n}$ obtained from it by deleting its top $n$ vertices (i.e., $TLR_{m,n}$ is obtained from $AR_{m,n}$ by deleting the $m$ leftmost and $n$ topmost vertices).

Let $a,b,c,d$ be positive integers. Note that $TLR_{c,d}$ has $d$ vertices on top (just like $LR_{c,d}$). In analogy to (\ref{Lgrapheq1}), if  $b+1\leq d$, define
\begin{equation}
\overline{L}^{\,a,b}_{c,d}:=\AR_{a-\frac12,b}\#TLR_{c,d},
\end{equation}
where the connected sum identifies the $i$th vertex on the bottom in $\AR_{a-\frac12,b}$ with the $i$th vertex on top in $LR_{c,d}$ -- but now counted from {\it right to left} -- for $i=1,\dotsc,b+1$ (the graph $\overline{L}^{\,3,3}_{3,6}$ is shown in Figure \ref{Lgraph}(c)).

Finally, if $b+1>d$, define $\overline{L}^{\,a,b}_{c,d}$ by
\begin{equation}
\overline{L}^{\,a,b}_{c,d}:=\AR^{**}_{a-\frac12,b}\#TLR_{c,d},
\end{equation}
where $\AR^{**}_{a-\frac12,b}$ is the graph obtained from  $\AR_{a-\frac12,b}$ by deleting the {\it leftmost} $b+1-d$ of its bottom vertices, and the connected sum identifies the $i$th vertex on the bottom in $\AR^{**}_{a-\frac12,b}$ with the $i$th vertex on top in $TLR_{c,d}$ (both counted from right to left), for $i=1,\dotsc,d$ ($\overline{L}^{\,4,6}_{2,4}$ is shown in Figure \ref{Lgraph}(d)).

To state our second lemma, it will be useful to introduce the following notation. Given an L-shaped graph $G$ from one of the two families defined above, define
$G_{\text{\rm bot}}$ to be the subgraph obtained from $G$ by removing all its bottommost vertices.

In the spirit of Section 3, if $G$ belongs to one of the above two families, or if it is obtained from a member of these families by deleting all its bottom vertices, we denote by ${}_|G$ (the ``combed version'' of $G$) the graph obtained from $G$ by appending a vertical edge to each of its bottommost vertices.

\begin{figure}\centering%
\includegraphics[width=14cm]{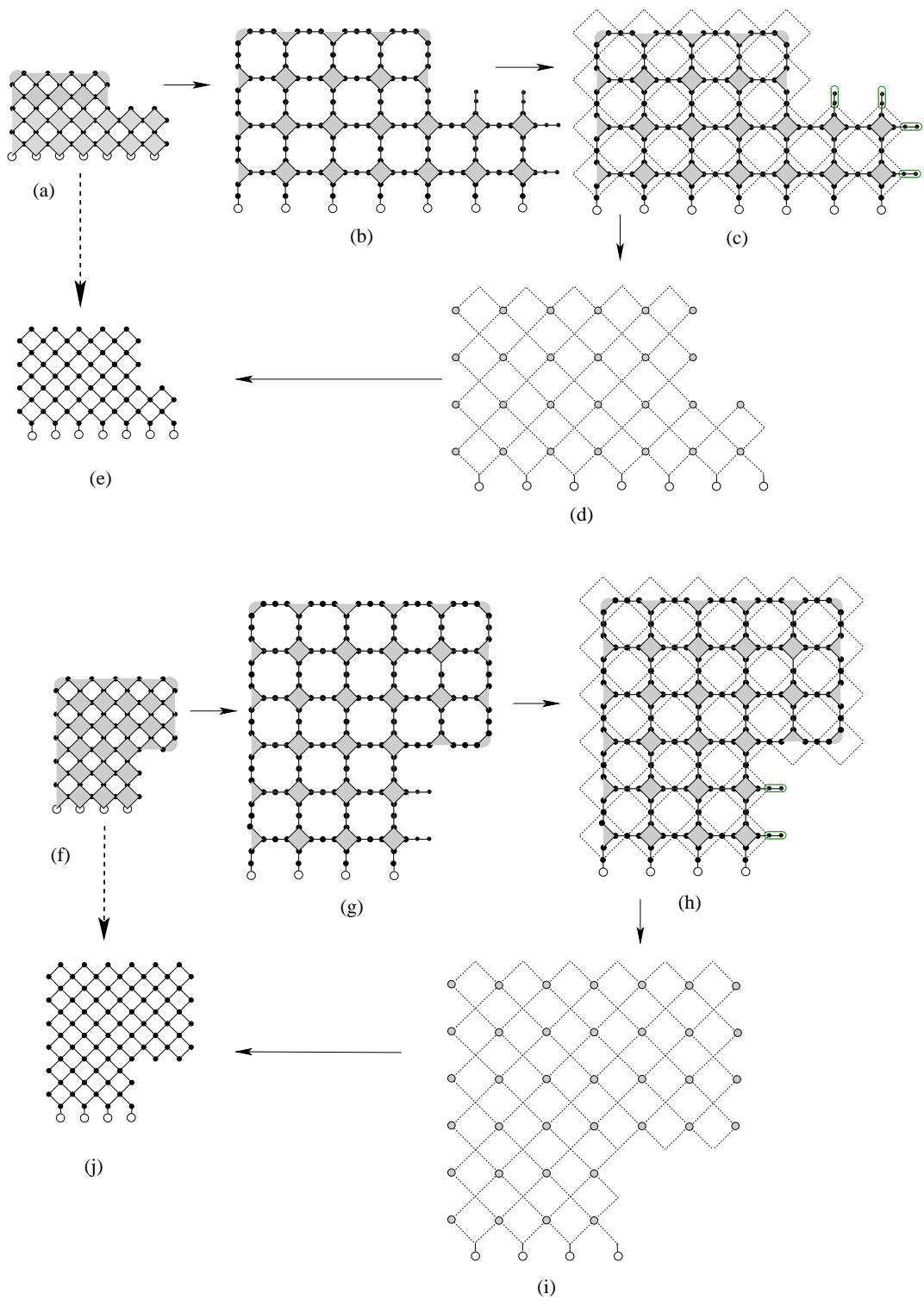}
\caption{Illustrating the L-transformations in Lemma \ref{Ltransform}(a). The dotted edges have weight $1/2$.}
\label{LemmaL}
\end{figure}

\begin{figure}\centering%
\includegraphics[width=12cm]{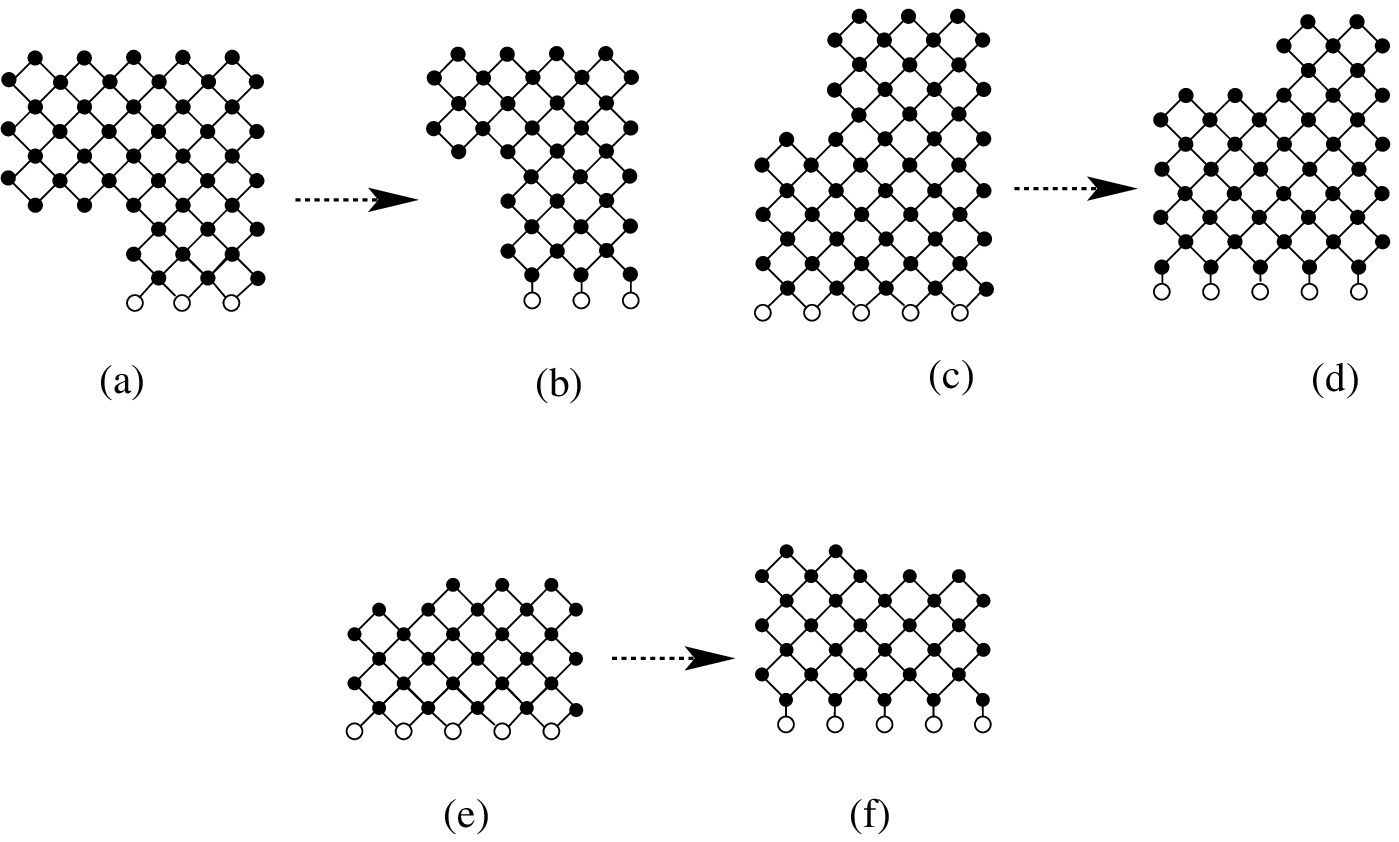}
\caption{Illustrating the L-transformations in Lemma \ref{Ltransform}(b) and (c).}
\label{LemmaL'}
\end{figure}

\begin{lem}[$L$-transformations]\label{Ltransform} Let $a$, $b$, $c$, $d$ be positive integers. Let $G$ be a graph, and let $U$ be an ordered $d$-element subset of the vertex set of $G$.

$(${\rm a}$)$. We have
\begin{equation}\label{Lgrapheq3}
\M(G\#L^{a,b}_{c,d})
=
2^{-a}
\M\left(G\#\, {}_{|}\!\!\left((L^{a+1,b+1}_{c,d})_{\text{\rm bot}}\right)\right).
\end{equation}
This transformation and its proof are illustrated in Figures \ref{LemmaL}(a)--(e), for $b+1\leq d$, and in Figures \ref{LemmaL}(f)--(j), for $b+1> d$.

$(${\rm b}$)$. If $a>1$, we have
\begin{equation}\label{Lgrapheq4}
\M(G\#\overline{L}^{\,a,b}_{c,d})
=
2^{a-1}
\M\left(G\#\, {}_{|}\!\!\left((\overline{L}^{\,a-1,b-1}_{c+1,d})_{\text{\rm bot}}\right)\right)
\end{equation}
(see Figures \ref{LemmaL'}(a) and (b) for an example when $b+1\leq d$, and see Figures \ref{LemmaL'}(c) and (d), for $b+1> d$).

$(${\rm c}$)$. If $b<d$, we have
\begin{equation}\label{Lgrapheq5}
\M(G\#\overline{L}^{\,1,b}_{c,d})
=
\M\left(G\#\, {}_{|}\!\!\left((L^{1,d-b}_{c,d})_{\text{\rm bot}}\right)\right).
\end{equation}
This is illustrated in Figures \ref{LemmaL'}(e) and~(f).

In each of the equalities (\ref{Lgrapheq3}), (\ref{Lgrapheq4}) and (\ref{Lgrapheq5}), the connected sum acts on $G$ along $U$, and on the other two summands along their bottommost vertices, ordered from left to right.
\end{lem}


We call the portions of an asymmetric quasi-hexagon $H$ above $\ell$, below $\ell'$ and between $\ell$ and $\ell'$ \textit{the upper}, \textit{the lower} and \textit{the middle parts} of the region, respectively. We also use these terms for the corresponding parts of the dual graph $G$ of the region. It means the upper, the lower and the middle parts of $G$ are the dual graphs of the upper, the lower and the middle parts of the region $H$, respectively.  We are now ready to present the proof of our theorem.

\begin{proof} [Proof of Theorem \ref{asymodd}.] Write for short
\begin{equation*}
H=H_a(d_1,\dotsc,d_k;c_1,\dotsc,c_t;d'_1,\dotsc,d'_l).
\end{equation*}
Denote by
\begin{equation*}
G:=G_a(d_1,\dotsc,d_k;c_1,\dotsc,c_t;d'_1,\dotsc,d'_l)
\end{equation*}
the dual graph of $H$.

(a) Apply the composite transformation of Proposition \ref{composite}(b) to the upper part of $G$, and also to its lower part. This yields
\begin{equation}\label{asyproofeq1}
\M(H)=2^{C+C'-h(a+k)-h'(a'+l)}\M(H_{a+k-1}(2h-1;c_1,\dotsc,c_t;2h'-1)),
\end{equation}
where $a'$ is the length of the southeastern side of $H$.

The arguments that led to (\ref{maineq2}) still work, because the middle part of $G$  is balanced (i.e. it has the same number of vertices in the two color classes of a proper 2-coloring), and they give
\begin{equation}\label{asyproofeq2}
a+k=a'+l.
\end{equation}
Furthermore, the arguments that gave (\ref{maineq2'}) also work, because of the same reason (the middle part of $G$ is balanced). It follows that $G$ is balanced if and only if $h=h'$, which proves part (a).

(b) Assume that $h=h'$, and write for simplicity
\begin{equation*}
H_1=H_{a+k-1}(2h-1;c_1,\dotsc,c_t;2h-1).
\end{equation*}
Then by (\ref{asyproofeq2}), equation (\ref{asyproofeq1}) becomes
\begin{equation}\label{asyproofeq3}
\M(H)=2^{C+C'-2h(a+k)}\M(H_1).
\end{equation}

Consider first the first equality in (\ref{asymoddeq}), provided $h_0=0$. In this case $c_i=1$, for $1\leq i\leq t$. To prove this we use the same key idea in the proof of Theorem \ref{main} as follows. The lozenge hexagon $H_{h,a+m-n-h,h+t}$ can be viewed as an asymmetric quasi-hexagon, in which all the inter-diagonal distances $d_i$'s, $d'_j$'s and $c_s$'s are equal to one. By the same argument in the equality (\ref{maineq3}), we obtain 
\begin{equation}\label{asyproofeq3'}
\M(H_{h,a+m-n-h,h+t})=2^{-h(h-1)}\M(H_1).
\end{equation}
Therefore, the first equality in (\ref{asymoddeq}) follows from (\ref{asyproofeq3}) and (\ref{asyproofeq3'}).

Consider next  the second equality in (\ref{asymoddeq}), provided $h_0>0$. By (\ref{asyproofeq3}), to prove this case it suffices to show that
\begin{equation}\label{asyproofeq4}
\M(H_1)=2^{h(h-1)-h_0(h_0-1)}\M\left(\Gamma_{a+k-h,h_0+2h+t,h}^{h_0,h_0+h-1}\right).
\end{equation}

One readily checks that, for $H_1$, the expression $C+C'-h(2a+2k-h+1)$ in the exponent of 2 in the second equality of (\ref{asymoddeq}) becomes precisely $h(h-1)$, and thus (\ref{asyproofeq4}) follows from the case $k=l=1$ of the second equality in (\ref{asymoddeq}).


Assume therefore that $k=l=1$. We will use the graph transformations presented earlier in this section to transform the graph $G$ into $\Gamma_{a+1-h,h_0+2h+t,h}^{h_0,h_0+h-1}$. We divide this process into four stages.

\begin{figure}\centering
\includegraphics[width=10cm]{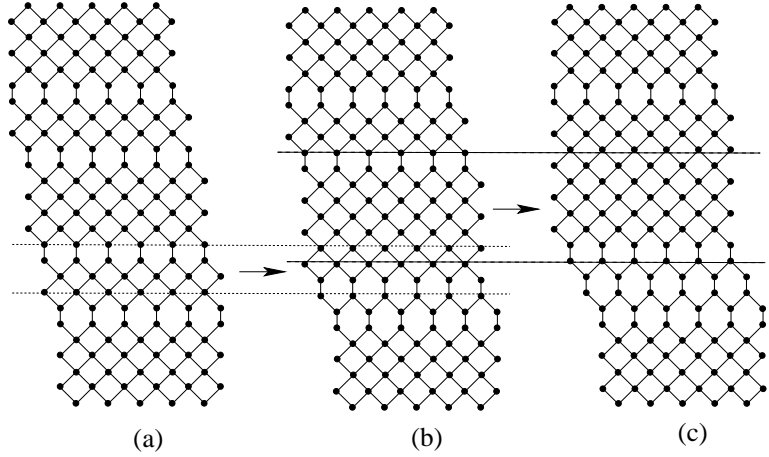}
\caption{The process in Stage 1.}
\label{middle7}
\end{figure}

\textbf{ Stage 1.} \textit{Reduction to the case $c_2=\cdots=c_t=1$.} If $c_t>1$ (and $t>1$), apply the transformation $T_3$ of Lemma \ref{T3} to the graph $G$, regarded as being obtained as a connected sum of type (\ref{T3eq1}), in which the second summand is the portion of $G$  that corresponds to the portion of $H$ in between $\ell'$ and the first diagonal above it (see Figure \ref{middle7}; we replace the subgraph between the two dotted lines in a graph by the subgraph between those lines in the next graph). One readily sees that the resulting graph is isomorphic to the graph
\begin{equation*}
G_1:=G_a(2h-1;c_1,\dotsc,c_{t-2},c_{t-1}+c_t-1,1;2h-1).
\end{equation*}
By Lemma \ref{T3}, we have $\M(G)=\M(G_1)$. If $c_t=1$, then $G\equiv G_1$ and the latter equality is obviously true. Repeating this argument, we get that $\M(G_i)=\M(G_{i+1})$ for $i=1,\dotsc,t-2$, where
\begin{equation*}
G_i:=G_a(2h-1;c_1,\dotsc,c_{t-i-1},c_{t-i}+\cdots+c_t-i,\underbrace{1,\dotsc,1}_{i};2h-1).
\end{equation*}
Thus
\begin{equation}\label{asyproofeq5}
\M(G)=\M(G_1)=\cdots\M(G_{t-1}).
\end{equation}


\begin{figure}\centering
\includegraphics[width=6cm]{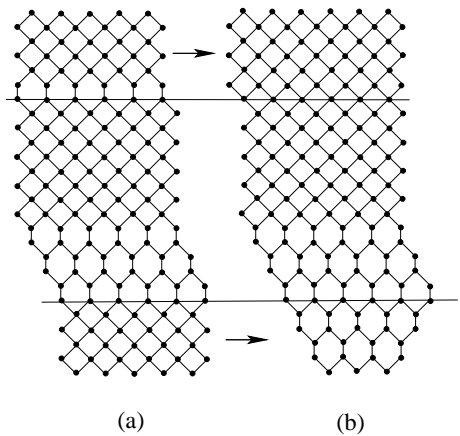}
\caption{The transformation in Stage 2.}
\label{Stage2}
\end{figure}

\textbf{Stage 2.}  \textit{Transforming $G$ into the connected sum of an Aztec rectangle with missing vertices and a half-honeycomb.} Apply the transformation $T_2$ of Lemma \ref{T2} to the part of $G_{t-1}$ above $\ell$, and then apply (in reverse) the composite transformation of Proposition~\ref{composite}(b) to replace the part of $G$ below $\ell'$ by a half-honeycomb (these transformations are illustrated in Figure \ref{Stage2}). Denote the resulting graph by $G_t$. The quoted results imply that
\begin{equation}\label{asyproofeq6}
\M(G_{t-1})=2^{-h}2^{h(h-1)/2}\M(G_t).
\end{equation}
Note that $G_t$ is not far from being a $\Gamma$-type graph --- the only difference is in how its upper part relates to an Aztec rectangle. This difference is eliminated in the remaining two stages.

\begin{figure}\centering
\includegraphics[width=12cm]{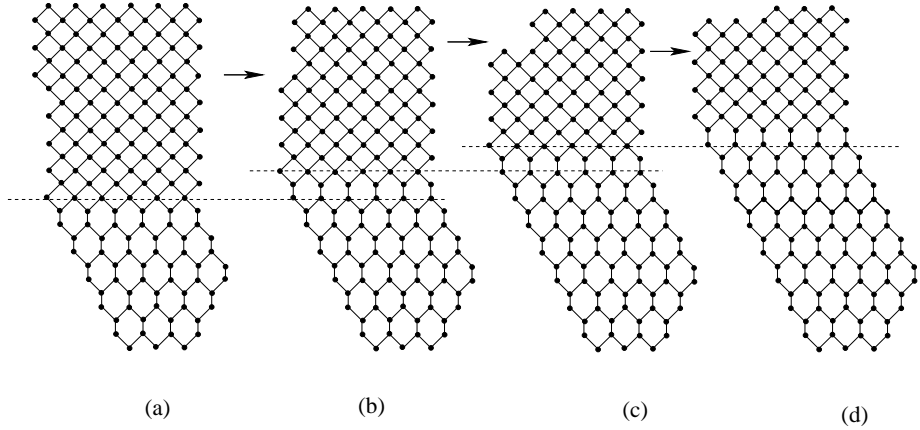}
\caption{The process in Stage 3.}
\label{middle8}
\end{figure}

\textbf{ Stage 3.}  \textit{Bringing the missing vertices of the Aztec rectangle to the top left. } This stage is illustrated in Figures \ref{middle8}(a)--(d); we replace the subgraph above the dotted line in a graph by the subgraph above this line in the next graph. Note that we have the graph isomorphism
\begin{equation}\label{asyproofeq7}
G_t\simeq\overline{L}^{\,h+1,a+1}_{h_0,a+1}\#B_{a+1-h,h+t,h}
\end{equation}
(recall that we are assuming $k=l=1$).

Apply Lemma \ref{Ltransform}(b) to the graph $G_t$ regarded as a connected sum as shown on the right hand side of (\ref{asyproofeq7}). We obtain that
\begin{equation}\label{asyproofeq8}
\M(G_t)=2^{h}\M(G_{t+1}),
\end{equation}
where
\begin{equation*}\label{asyproofeq9}
G_{t+1}\simeq\overline{L}^{\,h,a}_{h_0,a+1}\#B_{a+1-h,h+t+1,h};
\end{equation*}
shown in Figures \ref{middle8}(a) and (b).

Continue applying Lemma \ref{Ltransform}(b), and let $G_{t+i+1}$ be the graph resulting this way from $G_{t+i}$, for $i=1,\dotsc,h-1$. This yields
\begin{equation}\label{asyproofeq10}
\M(G_t)=2^{h}\M(G_{t+1})=2^{h}2^{(h-1)}\M(G_{t+2})=\cdots=2^{h(h+1)/2}\M(G_{t+h}),
\end{equation}
where
\begin{equation*}\label{asyproofeq11}
G_{t+h}\simeq\overline{L}^{\,1,a+1-h}_{h_0,a+1}\#B_{a+1-h,2h+t,h}.
\end{equation*}

\begin{figure}\centering
\includegraphics[width=14cm]{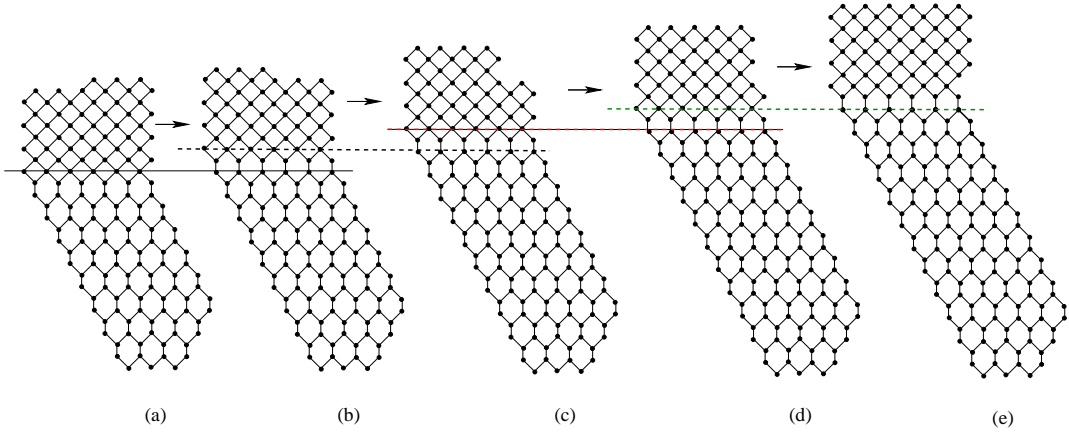}
\caption{The process in Stage 4}
\label{midnew3}
\end{figure}

\textbf{Stage 4.} \textit{Moving the missing vertices of the Aztec rectangle to the bottom right.} This stage is illustrated in Figures \ref{midnew3}(a)--(e). Apply Lemma \ref{Ltransform}(c) to the graph $G_{t+h}$. We obtain that
\begin{equation}\label{asyproofeq12}
\M(G_{t+h})=\M(G_{t+h+1}),
\end{equation}
where
\begin{equation*}\label{asyproofeq13}
G_{t+h+1}\simeq L^{1,h}_{h_0,a+1}\#B_{a+1-h,2h+t+1,h}
\end{equation*}
(see Figures \ref{midnew3}(a) and (b)).

Finally, apply Lemma 6.3(a) $h_0-1$ times to $G_{t+h+1}$ to obtain
\begin{equation}\label{asyproofeq14}
\M(G_{t+h+i})=2^{-i}\M(G_{t+h+i+1}),
\end{equation}
for $i=1,\dotsc,h_0-1$, where
\begin{equation*}\label{asyproofeq15}
G_{t+h+i}\simeq L^{i,h+i-1}_{h_0-i+1,a+1}\#B_{a+1-h,2h+t+i,h},
\end{equation*}
for $i=1,\dotsc,h_0$.

Note that for the last of these graphs we have
\begin{equation}\label{asyproofeq16}
G_{t+h+h_0}\simeq L^{h_0,h+h_0-1}_{1,a+1}\#B_{a+1-h,2h+t+h_0,h}\simeq \Gamma^{h_0,h_0+h-1}_{a+1-h,h_0+2h+t,h}.
\end{equation}
The second equality in (\ref{asymoddeq}) follows now by (\ref{asyproofeq5}), (\ref{asyproofeq6}), (\ref{asyproofeq10}), (\ref{asyproofeq12}), (\ref{asyproofeq14}) and (\ref{asyproofeq16}).

(c) From (\ref{asyproofeq3}), to prove part (c) we need to show that the region $H_1$ has no tiling when $a+k<h$, and that it has $2^{h(h-1)}$ tilings when $a+k=h$.

 Denote by $\widetilde{G}$ the dual graph of $H_1$. The top $2h-1$ rows of the vertices of $\widetilde{G}$ induce a subgraph isomorphic to $\AR_{h-1,a+m-n-1}$, and the bottom $2h-1$ rows of its vertices induce also a subgraph isomorphic to $\AR_{h-1,a+m-n-1}$. Denote by $\AR^{(1)}$ and $\AR^{(2)}$ the two Aztec rectangles (shown by the subgraph above and the subgraph below two dotted lines in Figure \ref{Ben2s}).  If $a+k<h$, then $\AR^{(1)}$ satisfies the conditions in Graph-splitting Lemma \ref{graphsplitting}(b) as an induced subgraph of $\widetilde{G}$, so $\M(\widetilde{G})=0$.

Assume that $a+k=h$, then two Aztec rectangles $\AR^{(1)}$ and $\AR^{(2)}$ are now two Aztec diamonds of order $h-1$. They satisfy the conditions in the part (a) of the Graph-splitting Lemma \ref{graphsplitting} as two induced subgraphs of $\widetilde{G}$ and $\widetilde{G}-\AR^{(1)}$, respectively.
\begin{subequations}
\begin{align}\label{power2b}
\M(\widetilde{G})&=\M(\AR^{(1)})\M(\widetilde{G}-\AR^{(1)})\\
&=\M(\AR^{(1)})\M(\AR^{(2)})\M(\widetilde{G}-\AR^{(1)}-\AR^{(2)})\\
&=2^{h(h-1)}\M(\widetilde{G}-\AR^{(1)}-\AR^{(2)}).
\end{align}
\end{subequations}
We can check that the graph $\widetilde{G}-\AR^{(1)}-\AR^{(2)}$ has a unique perfect matching (it consists of all vertical edges in the middle part of $\widetilde{G}$ when $h_0=0$, and its pattern is similar to the set of circled edges in Figure \ref{Ben2s} when $h_0>0$). Therefore,  we get $\M(\widetilde{G})=2^{h(h-1)}$ from (\ref{power2b}). This completes the proof for part (c).
\end{proof}

\begin{figure}\centering
\begin{picture}(0,0)%
\includegraphics{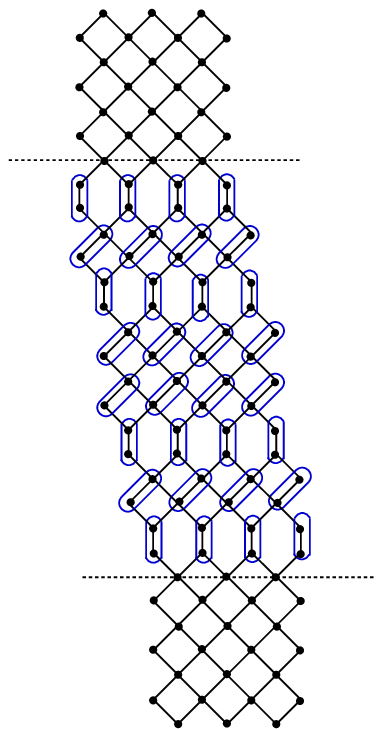}%
\end{picture}%
\setlength{\unitlength}{3947sp}%
\begingroup\makeatletter\ifx\SetFigFont\undefined%
\gdef\SetFigFont#1#2#3#4#5{%
  \reset@font\fontsize{#1}{#2pt}%
  \fontfamily{#3}\fontseries{#4}\fontshape{#5}%
  \selectfont}%
\fi\endgroup%
\begin{picture}(1927,3460)(1687,-4369)
\put(1702,-1198){\makebox(0,0)[lb]{\smash{{\SetFigFont{12}{14.4}{\rmdefault}{\mddefault}{\updefault}{$\AR^{(1)}$}%
}}}}
\put(2086,-4051){\makebox(0,0)[lb]{\smash{{\SetFigFont{12}{14.4}{\rmdefault}{\mddefault}{\updefault}{$\AR^{(2)}$}%
}}}}
\end{picture}
\caption{}
\label{Ben2s}
\end{figure}

The main reason we chose to relate quasi-hexagons to the graphs $\Gamma_{a,b,c}^{d,e}$ as in Theorem~\ref{asymodd} is that the number of perfect matchings of the latter can be written as a sum of quantities expressed by simple product formulas. This is presented in Corollary \ref{sumformulas} below, whose proof employs the following two lemmas.

The result below is due to Helfgott and Gessel (see Lemma 3 in \cite{Gessel}; see also \cite{Ciucu1}, (4.4) for a close analog).

\begin{lem}\label{Aztecdent} Label the bottom vertices of the baseless Aztec rectangle $\AR_{m-\frac12,n}$ from left to right by $1,\dotsc,n+1$, and denote by $G$ the graph obtained from it by deleting the vertices with labels in the set $\{t_1,\dotsc,t_m\}$, where $1\leq t_1<\cdots<t_{m}\leq n+1$ are given integers. Then we have
\begin{equation}
\M(G)=2^{\binom{m}{2}}\prod_{1\leq i<j\leq m}\frac{t_j-t_i}{j-i}.
\end{equation}
\end{lem}


The next result is due to Cohn, Larsen and Propp (see \cite{Cohn}, Proposition 2.1), see also Lemma 2 in \cite{Gessel}.

\begin{lem}\label{Hexdent} Consider the dual of the bottom half of a lozenge hexagon  of side-lengths $b$, $a$, $a$, $b$, $a$, $a$ (clockwise from top) on the triangular lattice. Label its topmost vertices from left to right by $1,\dotsc,a+b$, and the number of perfect matchings of the graph obtained from it by removing the vertices with labels in the set $\{r_1,\dotsc,r_{a}\}$ is equal to
\begin{equation}\label{hexdenteq}
\prod_{1\leq i<j\leq a}\frac{r_j-r_i}{j-i},
\end{equation}
where $1\leq r_1<\cdots<r_{a}\leq a+b$ are given integers.
\end{lem}

Denote by $V_{a,b}(r_1,\dotsc,r_{a})$ the product in (\ref{hexdenteq}), where $1\leq r_1<\cdots<r_{a}\leq a+b.$

\begin{cor}\label{sumformulas} $(${\rm a}$)$. If $e\neq c+d-1$, then $\M(\Gamma_{a,b,c}^{d,e})=0$.

 $(${\rm b}$)$. If $d\leq a$, then
\begin{equation}
\M\left(\Gamma_{a,b,c}^{d,c+d-1}\right)=2^{\binom{d}{2}}\sum V_{b,a}(A\cup\{a+c+1,\dots,a+b\}) V_{d,c}(B),
\end{equation}
where the sum is taken over all pairs of disjoint sets $A$ and $B$ whose union is $\{1,\dots,c+d\}$ and whose cardinalities are given by $|A|=c$ and $|B|=d$.

 $(${\rm c}$)$. If $d< a$, then
\begin{align}
&\M\left(\Gamma_{a,b,c}^{d,c+d-1}\right)=2^{\binom{d}{2}}\notag\\
&\quad\times\sum V_{b,a}(A\cup\{a+c+1,\dots,a+b\}) V_{d,c}(B\cup\{c+a+1,\dots,c+d\}),
\end{align}
where the sum is taken over all pairs of disjoint sets $A$ and $B$ whose union is $\{1,\dots,c+a\}$ and whose cardinalities are given by $|A|=c$ and $|B|=a$.
\end{cor}

\begin{proof} Consider the vertices of $\Gamma_{a,b,c}^{d,e}$ that belong both to the baseless Aztec rectangle part and to the half-honeycomb part (these are the vertices on the dotted line in Figure \ref{Cfigure}). In any given perfect matching, some of them are matched upward, and the remaining ones are matched downward. Let $k$ be the number of vertices belonging to both parts, and partition the set of perfect matchings of $\Gamma_{a,b,c}^{d,e}$ into $2^k$ classes corresponding to all the possible choices for each of these vertices to be matched upward or downward. Each class is then the set of perfect matchings of a disjoint union of two graphs, the top one being of the kind in Lemma \ref{Aztecdent}, and the bottom one (after adding some vertical forced edges that are shown as circled ones in Figure \ref{Cfigure}) being of the kind in Lemma \ref{Hexdent}. Part (a) follows from the requirement that these (bipartite) graphs are balanced, while parts (b) and (c) follow from Lemmas \ref{Aztecdent} and \ref{Hexdent}.
\end{proof}

We end this section by presenting the general version of Theorem \ref{asymodd}, corresponding to the case when the distances between successive diagonals are arbitrary (and not necessarily odd). To state the general result we need some more definitions.

We define the upper and the lower heights of the region to be the number of rows of regular black cells above $\ell$ and the number of rows of regular white cells below $\ell'$, respectively. The diagonals divide the middle part of an asymmetric quasi-hexagon into parts, called \textit{middle layers}.

\begin{rmk}
In the case of odd distances (as in Theorem \ref{asymodd}), the bottom row of cells is black. However, we face the same problem in investigating symmetric quasi-hexagons: The bottom row of cells may be white. By the same argument that we used in the proof of Theorem \ref{main}(a), when the latter happens the region has no tilings. Thus, we assume from now on that the bottom row of cells is black.
\end{rmk}

\begin{figure}\centering
\includegraphics[width=11cm]{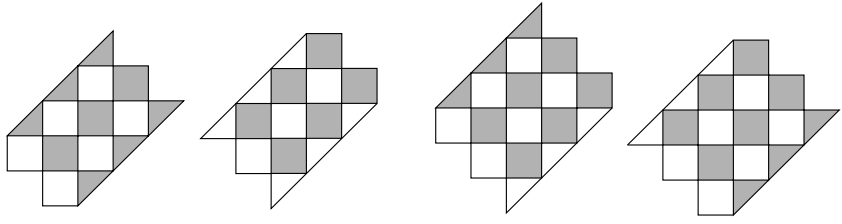}
\caption{From the left: a left-even middle layer, a right-even middle layer, a left-odd middle layer, and a right-odd middle layer.}
\label{middle1}
\end{figure}

Let $H:=H_{a}(d_1,\dotsc,d_k; c_1,\dotsc,c_t; d'_1,\dotsc, d'_l)$. Define the function
\[\phi_H:\{1,\dotsc,t\}\to\{-1,0,1\}\] by
\begin{equation*}
\phi_H(j)=
\begin{cases}
1& \text{\rm if $j$th middle layer is right-odd,}\\
-1& \text{\rm if $j$th middle layer is left-odd,}\\
1&\text{\rm otherwise,}
\end{cases}
\end{equation*}
where the four possible types of middle layers are illustrated in Figure \ref{middle1}.

Define the \textit{slope} of the asymmetric quasi-hexagon by
\begin{equation*}
\Phi(H)=\sum_{j=1}^t\phi_H(j).
\end{equation*}
Set
\begin{equation*}
h_0=\sum_{j=1}^t(c_i-\phi_H(j))/2.
\end{equation*}
Since the middle part is balanced, one can verify that the numbers of black and white vertices in the dual graph of $H$ are equal if and only if the upper and lower heights are equal. The same arguments that we used to prove Theorem \ref{asymodd} can be used to obtain the following result.

\begin{thm}\label{asymmetric1} Let $H:=H_a(d_1,\dotsc,d_k;c_1,\dotsc,c_t;d'_1,\dotsc,d'_l)$ be an asymmetric quasi-hexagon. Denote by $C$ the number of regular black cells above $\ell$, and by $C'$ the number of regular white cells below $\ell'$, and assume the upper and lower heights of $H$ are both equal to $h$. Let $m$ be the number of rows of regular black triangular cells above $\ell$, and $n$ be the number of rows of irregular black triangular cells above $\ell$. Then we have
\begin{align}
&\M(H)=2^{C+C'-h(2q-h+1)} \notag\\
&\quad\times
\begin{cases}
2^{-h_0(h_0-1)/2}\M\left(\Gamma^{h_0,h+h_0-1}_{q-h,h_0+\Phi(H)+2h,h}\right),&\text{if $h_0>0$ and $q>h$,}\\
\M(H_{h,q-h,h+t}),&\text{if $h_0=0$ and $q>h$,}\\
1, &\text{if  $q=h$,}\\
0, &\text{if  $q<h$,}
\end{cases}
\end{align}
where $q=a+m-n$.
\end{thm}

\thispagestyle{headings}

\end{document}